\documentclass[10pt,reqno]{amsart}
\usepackage{bbm}
\usepackage{amsmath,amsthm,amsfonts,amssymb,amscd,mathrsfs}
\usepackage{amssymb,amsfonts,amsthm, color, algorithm}
\usepackage{bm}
\usepackage{psfrag,graphicx}
\usepackage{epic,eepic}
\usepackage{color}
\usepackage{ebezier}
\usepackage{epsfig}
\usepackage{epstopdf}

\theoremstyle{plain}
\newtheorem{theorem}{Theorem}[section]

\newtheorem{claim}[theorem]{Claim}
\newtheorem{lemma}[theorem]{Lemma}
\newtheorem{cor}[theorem]{Corollary}

\newtheorem{pro}[theorem]{Proposition}

\numberwithin{equation}{section}

\long\def\begcom#1\endcom{}

\newcommand{\diam}{\operatorname{diam}}

\newcommand{\length}{\operatorname{\length}}

\newcommand{\qihao}{\fontsize{5pt}{\baselineskip}\selectfont}

\def\p{\prime}

\def\cP{\operatorname{\mathcal P}}

\def\cB{\operatorname{\mathcal B}}
\def\cM{\operatorname{\mathcal M}}

\def\cW{\operatorname{\mathcal W}}

\def\S{\operatorname{\mathbb S}}
\def\N{\operatorname{\mathbb N}}
\def\R{\operatorname{\mathbb R}}
\def\Ga{\operatorname{\Gamma}}
\def\ga{\operatorname{\gamma}}
\def\alp{\operatorname{\alpha}}
\def\Om{\operatorname{\Omega}}
\def\eps{\operatorname{\epsilon}}

\def\supp{\operatorname{supp}}

\def\length{\operatorname{length}}

\def\ln{\operatorname{ln}}

\def\top{\operatorname{top}}

\def\dim{\operatorname{dim}}
\def\Jac{\operatorname{Jac}}
\def\vep{\operatorname{\varepsilon}}
\begin{document}

\title[Continuity properties of folding entropy]{Continuity properties of folding entropy}

\author[G. Liao]{Gang Liao}
\address{G. Liao: School of Mathematical Sciences,  Center for Dynamical Systems and Differential Equations, Soochow University,
Suzhou 215006, China} \email{lg@suda.edu.cn}

\author[S. Wang]{Shirou Wang}
\address{S. Wang: Department of Mathematical \& Statistical Sciences,
University of Alberta, Edmonton, Alberta, Canada T6G 2G1}
\email{shirou@ualberta.ca}



\maketitle

\begin{abstract} 
The {\it folding entropy} is a quantity 
originally proposed by Ruelle in 1996  during the study of entropy production in the non-equilibrium statistical mechanics  \cite{Ruelle96}.  
As derived through a limiting process to the non-equilibrium steady state, the continuity of 
entropy production plays a key role in its physical interpretations.  
In this paper, we study the continuity  of folding entropy for a general  (non-invertible) differentiable dynamical system with degeneracy. By 
introducing a notion called {\it degenerate rate}, 
we prove that on any subset of measures with uniform degenerate rate, the folding  entropy, and hence the entropy production, is upper semi-continuous.  This extends the  upper semi-continuity result in \cite{Ruelle96} from endomorphisms to all $C^r(r>1)$ maps.  

We further apply to the one-dimensional setting. In achieving this, an equality between the folding entropy and (Kolmogorov-Sinai) metric entropy, as well as a general dimension formula are established. These admit their own interests. The upper semi-continuity of  metric entropy and dimension are then valid 
when measures with uniform degenerate rate are considered. 
Moreover, the sharpness of uniform degenerate rate is also investigated by examples in the scope of positive metric (or folding) entropy. 
\end{abstract}

\section{Introduction}

In the study of non-equilibrium statistical mechanics, 
the  statistical mechanical entropy is exhibited to be  persistently  pumped out of systems during  time evolutions  due to the energy or heat exchange with the environment. The numerical experiments in practice indicate   that 
the entropy production is non-negative  and  usually  positive in accordance with the second law of thermodynamics \cite{ECM90, ECM, GC}, and   such  phenomenon  in the mathematical representations as well has been  effectively  discussed  and justified to certain extent in the  stochastic process theory \cite{ JQQ,  KhinBook, QQG, Seif2012, ZQQ}.

As a fundamental approach to interpret the thermodynamics,  the  entropy production theory  was  especially  developed in the  framework of dynamical systems \cite{G2004, GC, GR97, Ruelle96, Ruelle99}.
In \cite{Ruelle96},  Ruelle investigated the entropy production in the standard dynamical system $(M, f, \mu)$ where $M$ is a compact Riemannian manifold, the evolution $f:M\to M$  is either a diffeomorphism or a non-invertible differentiable map,  and  $\mu$ is the non-equilibrium steady state which 
 is normally  thought as an 
SRB (Sinai-Ruelle-Bowen)  measure \cite{GR97, Ruelle99}.  The entropy production with respect to  $\mu$, denoted as $e_f(\mu)$, was  illustrated  through a limiting process in which a quantity called {\it folding entropy} naturally arises.
Adopting the Shannon-type expression $S(\bar\rho)=-\int\rho(x)\log\rho(x)dx$ as the statistical mechanical entropy  for a probability measure $\bar\rho$ with density $\rho,$ the  entropy pumped out of the system (to keep the  energy fixed due to the non-conservative forces acting on the system) within one-time  step is
\begin{eqnarray*}
e_f(\bar\rho)=-[S( f\bar\rho)-S(\bar\rho)]:=F_f(\bar\rho)-\int_M\log|\Jac(D_xf)|\rho(x)dx,
\end{eqnarray*}
where $f\bar\rho$ denotes the image of $\bar\rho$ under $f$.
The emerging term $F_f(\bar\rho),$  
called the folding entropy of $\bar\rho$,   exactly expresses the complexities  as $f$ ``folds"  the different states, according to their $\rho$-weights (or -masses),  into a same one.
By assuming  a general state 
$\mu$ as the  limit of a sequence of absolutely continuous measures with probability  densities $\{\rho_n\}_{n\ge1},$ the entropy production with respect to $\mu$
is intuitively given  by
\begin{eqnarray*}
e_f(\mu)=F_f(\mu)-\int_M\log|\Jac(D_xf)|d\mu.
\end{eqnarray*}

Physically considering $\mu$ as an idealization of $\bar\rho_n$ when $n$ large enough, a natural question is what is the relationship between $e_f(\mu)$ and the limiting quantity (if exists) $\lim_{n\to\infty}e_f(\bar\rho_n)$?  
Under some non-degenerate assumptions that essentially requires $f$ to be endomorphism-type, 
 Ruelle \cite{Ruelle96} showed that the folding entropy is upper semi-continuous and hence\footnote{Under the endomorphism-type assumption in \cite{Ruelle96}, $|\Jac(D_xf)|$ is uniformly away from $0$. Hence,  the term $\int\log|\Jac(D_xf)|d\mu$ is continuous with respect to $\mu.$}
\begin{eqnarray}\label{lim_proc}
e_f(\mu)\ge \limsup\nolimits_{n\to\infty}e_f(\bar\rho_n).
\end{eqnarray}
We remark that when $f$ is a diffeomorphism, the folding entropy is always zero (since no space-folding any more), and the entropy production 
is simply characterized by the phase volume contraction $-\int\log|\Jac(D_xf)|d\mu.$ This relation was earlier  
pointed in the discussion of more concrete non-equilibrium molecular dynamics models \cite{Evans85, ECM90,HP87}
and theoretically revealed 
in \cite{Andrey, GC}.
In this situation, the equality in the limiting process \eqref{lim_proc} naturally holds due to the continuity of the function $\log|\Jac(D_xf)|$.

By the  physical motivations to general situations,  the results in \cite{Ruelle96} were conjecturally extended and  a unified presentation was suggested.  The further progress actually depends on the study in the ergodic theory of differentiable dynamical systems. For a general  (non-invertible) system beyond  endomorphism,    the analysis difficulty of  entropy production  is  in the handling of the possible accumulation of ``foldings" due to the degeneracy (even with zero measure) in the approximation process. 
In this paper, we assume, standardly,  that $f$ is a $C^{r} (r>1) $ map  on a compact Riemannian manifold  $M$. Our main goal is to explore the  mechanism 
for the upper semi-continuity of folding  entropy. 
By introducing a notion called  {\it degenerate rate}  which captures the complexity arising near  the degenerate set $\Sigma_f:=\{x\in M: \Jac(D_xf)=0\}$,  we shall  establish the upper semi-continuity of folding entropy  
on any set of measures with uniform degenerate rate. This thus justifies \eqref{lim_proc}
in the limiting process.

\subsection{Folding entropy and degenerate rate} 
We begin with the precise definition of  folding entropy. Denote by $\cP(M)$ the set of all Borel probability measures on $M$. 
For $\mu\in\cP(M),$ 
the folding entropy $F_f(\mu)$ is a
{\it conditional entropy} measuring the   complexities of the $\mu$-weighted preimages of $f.$
More specifically, denote $f^{-1}\epsilon =\big\{\{f^{-1}x\}\big\}_{x\in M}$ as the preimage partition of $f$ with 
$\epsilon:=\{\{x\}\}_{x\in M}$ being the partition  into single points.  
Then the {\it folding entropy} of $f$ with respect to $\mu$    is the conditional entropy of $\epsilon$ relative to $f^{-1}\epsilon$:
\begin{eqnarray*}
F_f(\mu)=H_\mu(\epsilon|f^{-1}\eps)=\displaystyle\int_M H_{\tilde{\mu}_x}(\epsilon)d( f\mu),
\end{eqnarray*}
where  $\{\tilde{\mu}_{x}\}$ is the family of conditional measure of $\mu$ disintegrated along the preimage sets\footnote{Readers may refer to \cite{Rok} for rigorous mathematical definitions of the conditional measure and entropy.} $\{f^{-1}x\}$.

Despite for its seemingly abstract form, the folding entropy is in practice rather intuitive; see Figure \ref{fig:folding}(A).
We  see that if 
there are only finitely many preimage branches, say $N,$ then the folding entropy is always bounded by $\log N$ since $H(\tilde{\mu}_{x})\le \log N$, and  the equality is achieved if all the $N$ preimages are equally $\mu$-weighted. 
In particular, if $f$ is a diffeomorphism, then $N=1$ and hence $F_f(\mu)=0$. A simple but illuminating example  is to consider $M=\S^1$ and $f: x\mapsto Nx$ (mod 1) for $N\in\mathbb{N}.$ 
In this example, since $f$ ``evenly" expands the state space $\S^1$, the folding entropy with respect to the Lebesgue measure  (which is also an SRB measure) is $F_f(\mathcal{L}eb)=\log N.$
Generalizing a bit more  in the situation  of two preimage pieces,  consider $f(x)=px$ for $x\in[0,1/p)$ and $f(x)=\frac{p}{p-1}(1-x)$ for $x\in[1/p,1)$,  where $1<p\neq 2$.  
The subintervals  $[0,1/p)$ and $[1/p,1)$ are stretched by $f$ in different scales. 
Still, the Lebesgue measure is an SRB measure, however the folding entropy 
$F_f(\mathcal{L}eb)=-\frac{1}{p}\log\frac{1}{p}-\frac{p-1}{p}\log\frac{p-1}{p}<\log2.$

As a consequence of physically motivated and mathematically intuitive role played by  the folding entropy,  a folding-type  entropy inequality  on the backward process was simultaneously revealed  in   \cite{Ruelle96}  in contrast with the forward evolution.
To be specific, for a $C^r(r>1)$ map $f$ on a compact Riemannian manifold  $M$ and an $f$-invariant measure $\mu,$  
the folding-type Ruelle inequality reads as 
\begin{eqnarray}\label{folding_Ruelle_ineq}
h_\mu(f)\le F_f(\mu)-\int_M\sum_{\lambda_i(f,x)<0}\lambda_i(f,x)d\mu,
\end{eqnarray}
where $\lambda_i(f,x)$ denote the Lyapunov exponents,  existing for $\mu$-a.e. $x\in M$ by the Oseledec theorem \cite{O68}.  The inequality \eqref{folding_Ruelle_ineq}   was conjectured in  \cite{Ruelle96}  providing 
an approach to the non-negativity of entropy production and  was mathematically rigorously proved  by Liu \cite{Liu03} for $C^r(r>1)$ maps under some polynomial-like degeneracy  conditions, and  then by the authors  \cite{LW}   for all $C^r(r>1)$ maps.  
We would like to mention  that in the classical Margulis-Ruelle inequality \cite{Ruelle78},   the RHS  is simply the sum of positive Lyapunov exponents \cite{Ruelle78}. With respect to  the backward process, however, since one trajectory may split into many branches, 
the folding entropy
 appears as a characterization of  this ``global" expansion in terms of branches, while the ``local" expansions inside each branch are captured by the Lyapunov exponents.

\begin{figure}[ht] 
\includegraphics[height=4.5cm]{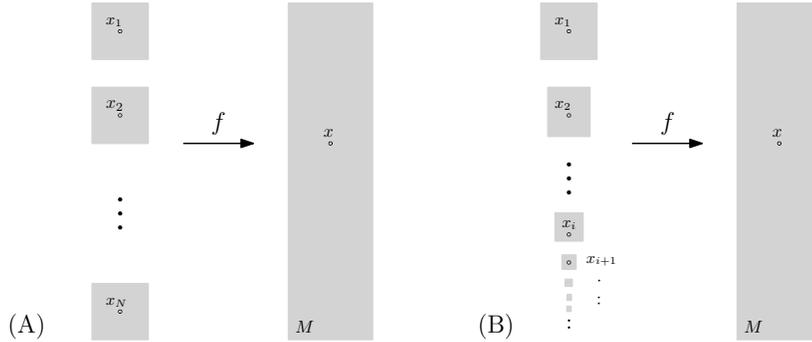}
\caption{(A) There are $N$ branches containing  the $N$ preimages, $x_1,\cdots, x_N,$  of a typical $x\in M,$ i.e., $f(x_1)=\cdots=f(x_N)=x.$ 
\ (B) There are countably many branches containing the preimages, $x_1,\cdots, x_i, \cdots,$ of a typical $x\in M.$ The accumulation happens near the degenerate set. 
}
\label{fig:folding}
\end{figure}

In \cite{Ruelle96},  it was shown  that under the ``endomorphism-type" assumptions\footnote{ In 
\cite{Ruelle96}, it was assumed that 
the state space (excluding the degenerate set) $M\backslash\Sigma_f$ can be divided into several pieces which have either  identical or disjoint images and  the Jacobian are uniformly away from zero. 
This rules out the possibility of 
accumulations of infinite ``foldings" occurring near the degenerate set during the approximation process.},  the folding entropy varies in an upper semi-continuous way. In this situation, a single  coarse-grained partition alone can exhaust all the complexities in the folding process. However, this does not work in general,  especially when arbitrarily many preimage branches accumulate around the degenerate set; see Figure \ref{fig:folding}(B).
A   general influence mechanism 
by the degeneracy for systems beyond  endomorphisms remains unknown yet.  The main goal of this paper is to develop a such mechanism.
To achieve this, we introduce a notion called {\it degenerate rate} to capture the emerging 
 complexities in the refining coarse-graining  process. 
Take  firstly  a decreasing sequence of 
open neighborhoods of $\Sigma_f,$ denoted by $\mathcal{V}=\{V_m\}_{m\ge1},$ such that  
\begin{eqnarray}\label{open_seq}
d_H(V_m, \Sigma_f)\to 0,\quad m\to\infty,
\end{eqnarray}
 where $d_H$ denotes the Hausdorff distance. Let $\bar\eta=\{\eta_{m}\}_{m\ge1}$ be any sequence of positive real numbers 
approaching to zero. 
  Given $\mu\in\cP(M),$ $f$ is said to  admit \emph{degenerate rate} $\bar\eta$  with respect to $\mu$ if
\begin{eqnarray*}\label{def:uniform-degenerate-rate}
\Big{|}\int_{V_m}\log|\Jac(D_xf)| d\mu\Big{|}\leq\eta_{m}, \quad 
\forall m\ge1. 
\end{eqnarray*}
In the definition,   $\bar\eta$ actually  characterizes the level of complexities  arising near
$\Sigma_f$,  i.e.,  how  $|\Jac(D_xf)|$ diminishes on the shrinking domains around $\Sigma_f$. 
It is well-defined for
any $\mu\in\cP(M)$ satisfying the integrable condition\footnote{Observe that \eqref{finite LE} implies $\mu(\Sigma_f)=0$. Hence, the sequence $\{\eta_m:=|\int_{V_m}\log|\Jac(D_xf)| d\mu|\}_{m\ge1}$  approaching to zero is just right.} 
\begin{eqnarray}\label{finite LE}
\Big|\int\log|\Jac(D_xf)|d\mu\Big|<\infty.
\end{eqnarray}
We will  focus on the probability measures
with a {\it uniform} degenerate rate. Define 
\begin{eqnarray*}
\cP_{\bar \eta}(f):=\Big\{\nu\in\cP(M):\Big{|}\int_{V_m}\log|\Jac(D_xf)| d\mu\Big{|}\leq\eta_{m},\,\,\forall m\ge1\Big\},
\end{eqnarray*}
i.e., $\cP_{\bar \eta}(f)$ collects all the probability  measures with the degenerate rate $\bar\eta$. Note that any $\mu\in\cP_{\bar\eta}(f)$, $\mu(\Sigma_f)=0$  holds automatically.
 Before proceeding to the theorems, we give two remarks about 
$\cP_{\bar\eta}(f)$:

\medskip 

\noindent (i)
$\cP_{\bar\eta}(f)$ is a closed subset of $\cP(M)$ in the weak$^*$-topology. To see this, let $\{\mu_i\}_{i\ge1}$ be a sequence of measures in $\cP_{\bar \eta}(f)$  satisfying  $\lim_{i\to\infty}\mu_i=\mu$.
Without loss of generality, we assume that $|\Jac(D_xf)|\big|_{V_m}<1$ for all $m\ge1.$ Then for any $s>m$,  $|\int_{V_m\setminus V_s}\log|\Jac(D_xf)| d\mu_i|\le \eta_m.$ Hence,  \[\Big{|}\int_{V_m\setminus V_s}\log|\Jac(D_xf)| d\mu\Big{|}=\lim_{i\to \infty} \Big{|}\int_{V_m\setminus V_s}\log|\Jac(D_xf)|d\mu_i\Big{|}\le \eta_m.\]
This, together with $\mu(\Sigma_f)=0,$ yields 
\[\Big{|}\int_{V_m}\log|\Jac(D_xf)| d\mu\Big{|}\le \eta_m, \quad\forall m\ge1. \]
That is, $\mu\in\cP_{\bar\eta}(f).$

\medskip 

\noindent (ii) The uniform  degenerate rate is independent on the choice of $\mathcal{V}=\{V_m\}_{m\ge1}$.  Actually,  for any two sequences  of neighborhoods of $\Sigma_f,$ $\mathcal{V}^{(i)}=\{V_m^{(i)}\}_{m\ge1}$ $(i=1,2)$ satisfying \eqref{open_seq},  the associated tails
\[\Big\{\Big|\int_{V_m^{(i)}}\log|\Jac(D_xf)| d\mu\Big| \Big\}_{m\ge1}\quad (i=1,2),\]  are uniformly equivalent,  i.e.,   for any $\gamma>0$ and $m_1\ge1,$  there exists $m_2>0$ such that 
\[\Big|\int_{V_{m}^{(2)}}\log|\Jac(D_xf)| d\mu\Big|<\Big|\int_{V_{m_1}^{(1)}}\log|\Jac(D_xf)| d\mu\Big|+\gamma,\quad\forall m\ge m_2,\]  and vice versa.
Henceforth, we fix a  decreasing sequence $\mathcal{V}=\{V_m\}_{m\ge1}$ satisfying \eqref{open_seq}.  By the variation of the sequence $\bar\eta=\{\eta_{m}\}_{m\ge1}$  with zero limit,  all the settings of uniform degenerate rate will then be exhausted. 

\medskip

For a general  $C^{r}(r>1)$ system,   the upper semi-continuity of folding entropy  is   established  when measures with uniform degenerate rates are considered. 
\begin{theorem}\label{thm1}
Let $f$ be a $C^{r}(r>1)$ map  on a compact Riemannian manifold   $M.$  Then  the folding  entropy 
\begin{eqnarray*}
\mu\mapsto F_f(\mu)
\end{eqnarray*}
is upper semi-continuous  on $\cP_{\bar\eta}(f).$
\end{theorem}

Note that on  $\cP_{\bar \eta}(f),$ the map $\mu\mapsto\int \log| \Jac(D_xf)|d\mu$
is continuous because the singular part, i.e., the integral on the neighborhood of $\Sigma_f$, is uniformly controlled. 
Thus, Theorem \ref{thm1} directly yields
the upper semi-continuity of the entropy production.

\begin{cor}\label{cor1}
Let $f$ be a $C^{r}(r>1)$ map  on a compact  Riemannian manifold   $M.$  Let $\{\mu_n\}_{n\ge 1}\subseteq\cP_{\bar\eta}(f)$ be a sequence of measures such that $\lim_{n\to\infty}\mu_n=\mu.$
Then it holds that
\begin{eqnarray*}\label{lim_meas}
e_f(\mu)\ge \limsup\nolimits_{n\to\infty}e_f(\mu_n).
\end{eqnarray*}
\end{cor}

\subsection{Applications to  interval maps}
As is well-known that for a measurable dynamical system, the (Kolmogorov-Sinai) {\it metric entropy}  measures 
the complexities (or uncertainties) produced during the dynamical evolutions. Among the 
various properties of metric entropy  that has been well studied, the upper semi-continuity has drawn much attention as it plays a key role for the existence of equilibrium states \cite{BowenBook,  walters82}.  
In the context of interval maps, we will   establish an equality relation between folding entropy and  metric entropy; see Theorem \ref{thm:entropy equality}. Then as an application of Theorem \ref{thm1}, 
the upper semi-continuity property of the latter one is obtained 
when invariant measures with uniform degenerate rates are considered. Henceforth,  for an interval map $f$,  let $\cM_{inv}(f)$ be the set of all  $f$-invariant Borel probability measures, and still,   $\bar\eta=\{\eta_{m}\}_{m\ge1}$ be the any sequence of positive real numbers approaching to zero. 
Then \[\mathcal M_{\bar \eta, inv}(f):=\cP_{\bar \eta}(f)\cap \mathcal{M}_{inv}(f)\] denotes the set of all $f$-invariant  Borel probability  measures with uniform degenerate rate $\bar\eta$. 

\begin{theorem}\label{thm2}
Let $f$ be a $C^{r}(r>1)$ interval map. 
Then the  metric entropy 
\begin{eqnarray*}
\mu\mapsto h_\mu(f)
\end{eqnarray*}
is upper semi-continuous on $\mathcal M_{\bar \eta, inv}(f).$
\end{theorem}

\noindent{\it Remark:}  
In the theory of differentiable  dynamical systems, the hyperbolicity and  the smoothness are two basic mechanisms leading to the upper semi-continuity of metric entropy.   On the one hand, 
for all $C^{\infty}$ systems, the upper semi-continuity of metric (or topological  entropy) always holds true \cite{New89, Yomdin}.  In fact, $C^{\infty}$ smoothness brings about the uniform control of degeneracy and hence foldings, which, as already shown by various examples \cite{Buzzi, Mis71, Mis73, Ruette},   can be destroyed if only finite smoothness is considered.  On the other hand,  for systems with some  hyperbolicity, 
e.g., the uniformly hyperbolic systems \cite{Bowen72}, the partially hyperbolic systems with one dimensional center \cite{CY, DFPV}, and the diffeomorphisms away from tangencies \cite{LVY},
the upper semi-continuity  property of metric entropy is shown to be valid.  In the particular setting of non-uniformly hyperbolic systems,  Newhouse proposed  hyperbolic rate to characterize  the level of hyperbolicity  for invariant measures, and proved the upper semi-continuity of  metric entropy  on the set with uniform hyperbolic rate \cite{New89}.  The  degenerate rate in this paper for the studying of folding entropy, without any requirements of  measure preserving and hyperbolicity, 
plays an analogous role of the hyperbolic rate from the view of the degeneracy.


\medskip

Besides the entropy,   dimension is another intrinsic characteristic of complexity in the dynamical theory. 
For a differentiable dynamical system, 
the interrelation among the various  dynamical quantities has been intensively
investigated \cite{ Led,   LY, Pesin, Ruelle78, Young82}. In particular, for an interval map $f$ and an ergodic $f$-invariant Borel  probability measure $\mu,$ 
the  following formula
\begin{eqnarray}\label{relation}
h_\mu(f)=\lambda^+(\mu)\dim(\mu)
\end{eqnarray}
was established in \cite{Led} under certain assumptions on the degeneracy, where $\lambda^+(\mu)=\max\{\lambda(\mu), 0\}$ is the positive part of the Lyapunov exponent $\lambda(\mu)$ of $\mu$.  In Section \ref{Application}, by developing an integrable version of the Brin-Katok formula, 
we show the validity of \eqref{relation} for any $C^r(r>1)$ interval map $f$ and {\it hyperbolic} (i.e., $\lambda(\mu)\neq0$) ergodic Borel probability measure $\mu$; see Theorem \ref{thm:dim_formula}. We emphasize that the hyperbolicity of measure is necessary here, since the dimension may not exist for non-hyperbolic measures in general \cite{KS, LM}.  
As a consequence of Theorem \ref{thm2}, we obtain the upper semi-continuity of dimension at all  hyperbolic measures on the set of  ergodic measures with uniform degenerate rate. 
For convenience, write \[\mathcal M_{\bar \eta, erg}(f)=\mathcal P_{\bar \eta}(f)\cap \mathcal{M}_{erg}(f),\] 
where  $\cM_{erg}(f)$ denotes the set of all ergodic $f$-invariant Borel probability measures.

\begin{theorem}\label{thm3}
Let $f$ be a $C^{r}(r>1)$ interval map. Then the  dimension 
\begin{eqnarray*}
\mu\mapsto \dim(\mu)
\end{eqnarray*}
 is upper semi-continuous  at all hyperbolic  measures in $\mathcal M_{\bar \eta, erg}(f).$
\end{theorem}

\subsection{Sharpness of the  uniform degenerate rate condition}
For the previous known examples concerning the loss of upper semi-continuity of metric entropy, the ergodic measures at which the  upper semi-continuity fails are all atomic and hence admit zero metric entropy  (see, for instance,  \cite{Buzzi, Mis73, Ruette}).  Note  that the positivity of metric entropy for an ergodic measure implies that the  hyperbolicity holds  in a  nontrivial full-measure set  (due to the Margulis-Ruelle inequality \cite{Ruelle78} and ergodicity). In the consideration that the hyperbolicity  has been a possible  approach to  the upper semi-continuity of metric entropy \cite{Bowen72}, there arises a natural question suggested by  Burguet in \cite{Burguet}

\vspace{2mm}

\noindent{\it\bf Question:} Is the metric entropy  of a $C^r$ ($r> 1$) interval map   upper semi-continuous at ergodic measures with positive entropy?

\vspace{2mm}

Taking into account the approximation of measures by the generic orbits,   we will  construct a type  of modified  examples  in Section \ref{example} for which the loss of upper semi-continuity of metric entropy occurs at a non-atomic ergodic measure, say $\mu,$ with positive entropy.  The defect of upper semi-continuity of metric entropy at $\mu$ follows from the infinite ``foldings" around a homoclinic tangency  whose images in average converging to $\mu;$
see Figure \ref{fig:accumulation of small horseshoes}. The approximation is realized by taking some generic point of $\mu,$ say $x_0$, and then elaborately choose  a sequence of (ergodic) measures whose generic points follow the orbit of $x_0$  with large frequency, but do NOT admit a  uniform degenerate rate. 
 This indicates that 
in the setting of differentiable systems (with degeneracy), the uniform degenerate rate  serves as an essential condition for  the upper semi-continuity of folding entropy.

\begin{theorem}\label{thm4}
For any $1<r<\infty$, there exist   $C^r$ interval maps admitting  ergodic measures with positive entropy as the non upper semi-continuity points of the metric (or folding) entropy. 
\end{theorem}

This paper is organized as follows. In Section \ref{sec:preliminary}, we review the basic notions in the ergodic and entropy theory that will be used throughout this paper. In Section \ref{sec:degenerate}, 
we prove Theorem \ref{thm1}.  In Section \ref{Application},  we  discuss the applications 
to the  one-dimensional setting. 
An entropy formula (Theorem \ref{thm:entropy equality})  and  a dimension formula  (Theorem \ref{thm:dim_formula}) are established for all $C^r(r>1)$ interval maps. Theorem \ref{thm2} and Theorem \ref{thm3} then follows almost immediately. 
In Section \ref{example},  we concretely construct a type  of  modified interval maps for which
a typical failing mechanism of the upper semi-continuity of metric entropy at an ergodic measure with positive entropy is developed. 
The sharpness of the condition about the uniform degenerate rate in Theorem \ref{thm1} and Theorem \ref{thm2}  is also illustrated.

\section{Preliminaries}\label{sec:preliminary}
In this section, we recall  some basic notions and facts in the ergodic and entropy theory that will be used in the later discussions. Readers may refer to \cite{Rok, walters82} for more details.

\medskip

\medskip

\noindent{\it - Invariant  measure, partition and entropy.} 
Let  $X$ be a compact metric space  with  the Borel $\sigma$-algebra $\mathcal B$, and $f:X\to X$ be a measurable transformation.
Recall as in the introduction that $\cP(X)$ denotes the set of all  Borel probability measures on $X$ endowed with the  weak$^*$-topology, i.e.,  $\mu_n\to\mu$ in $\cP(X)$ if and only if for any  continuous function $\varphi$, $\int\varphi d\mu_n\to\int\varphi d\mu$ as $n\to\infty.$ 
For $\mu\in\cP(X),$
the {\it image} of $\mu$ under $f$ is given by  $ f\mu(B)=\mu(f^{-1}B), \forall B\in\cB,$ and $\mu$ is said to be {\it $f$-invariant} if $ f\mu=\mu.$ Denote by $\cM_{inv}(f)$  the set of all  $f$-invariant Borel probability measures.

 A  {\it partition} $\xi=\{A_{\alp}\}_{\alp\in\mathscr A}$ of $X$ 
is a collection of  disjoint elements of $\cB$ such that $\cup_{\alpha\in\mathscr A}A_{\alp}=X.$
In particular, $\xi$ is called {\it finite} if  the  cardinality  $\#\mathscr A<\infty.$
Given two partitions $\xi=\{A_\alpha\}_{\alpha\in\mathscr A}$ and $\zeta=\{C_\gamma\}_{\ga\in\mathscr C},$
the join of  $\xi$ and $\zeta$ is a partition of $X$, denoted as $\xi\vee\zeta,$ such that
\begin{eqnarray*}
\xi\vee\zeta=\{A_\alpha\cap C_\gamma\}_{\alpha\in\mathscr A,\gamma\in\mathscr C}.
\end{eqnarray*}
If $\xi\vee\zeta=\zeta,$ i.e., 
each element of $\xi$ is a union of elements of $\zeta,$ then 
we call $\zeta$ is a {\it refinement} of $\xi,$ 
and write $\xi\preceq\zeta.$
A sequence of partitions $\{\xi_n\}_{n\ge1}$
is said to be {\it refining} (or {\it increasing}) if $\xi_1\preceq\xi_2\preceq\cdots\preceq\xi_n\preceq\cdots.$

Given $\mu\in\cP(X)$ and a {\it finite} partition $\xi$ of $X,$  the {\it entropy} of $\xi$ with respect to $\mu$ is   
\begin{eqnarray*}
H_\mu(\xi)=\displaystyle\int_X -\log\mu(\xi(x))d\mu(x),
\end{eqnarray*}
where  $\xi(x)$ denotes the element of $\xi$ containing $x.$ 
For any two finite partitions $\xi$ and $\zeta$,  the {\it conditional entropy} of $\xi$ given $\zeta$ is 
\begin{eqnarray*}\label{eq:conditional-entropy}
H_\mu(\xi|\zeta)=\sum_{C\in\zeta}\mu(C) H_{\mu_C}(\xi|_C)
=\sum_{A\in\xi,C\in\zeta}-\mu(A\cap C)\log\Big(\dfrac{\mu(A\cap C)}{\mu(C)}\Big)
\end{eqnarray*}
where $\mu_C(A):=\mu(A\cap C)/\mu(C)$ is the conditional measure of $A$ given $C,$ and $\xi|_C$ denotes the partition $\xi$ restricted on $C.$
For any   $\mu\in \cM_{inv}(f)$,  the {\it metric entropy} of $f$ with respect to $\mu$ and $\xi$ is
\begin{eqnarray}\label{metric-entropy-by-conditional-entropy}
h_\mu(f,\xi)=\lim_{n\to\infty}H_\mu(\xi|\bigvee_{i=1}^n f^{-i}\xi)=\inf_{n\ge1}H_\mu(\xi|\bigvee_{i=1}^n f^{-i}\xi),
\end{eqnarray}
and the metric entropy of $f$  with respect to $\mu$  is 
\begin{eqnarray*}
h_\mu(f)=\sup_{\xi}h_{\mu}(f,\xi),
\end{eqnarray*}
where the supremum is taken over all finite partitions of $X.$

We remark that the entropy and conditional entropy can be defined in the more general setting where  $(X,\cB, \mu)$ is a Lebesgue space and $\xi,\zeta$ are measurable partitions. The folding entropy is a particular type of conditional entropy in this setting with $\xi$ and $\zeta$ be chosen as $\eps$ and $f^{-1}\eps,$ respectively.  See \cite{Rok} for more  discussions on this. 

In the estimation of entropy,   a   basically important function 
from the information theory is as follows: 
\begin{eqnarray*}
	\phi(x)=\begin{cases}
		-x\log x,&x\in(0,1];\\
		0,&x=0.
	\end{cases}
\end{eqnarray*} 
As a simple result of the concavity of $\phi$, the following proposition will be use several times in section \ref{sec:degenerate}. 
\begin{pro}\label{prop:phi}
	Let $\{p_i\}_{i=1}^n$ be such that $\sum_{i=1}^n p_i=1$ and   $p_i\ge0, 1\le  i\le n$, and $\{x_i\}_{i=1}^n \subset [0,1]$.    Then
	\[\sum_{i=1}^n p_{i}\phi(x_{i})\le\phi(\sum_{i=1}^np_i  x_i).\]
	In particular,   $\sum_{i=1}^n \phi(p_{i})\le\log n.$
\end{pro}

\medskip

\noindent{\it - Dimension of a measure.}
For  the compact metric space   $X$ with distance $d$,  given any finite Borel measure $\mu$, for $\mu$-a.e. $x\in X$, the  lower and upper local dimension are respectively defined as 
\begin{eqnarray*} 
	\underline{\dim}(x) =\liminf_{\delta\to0}\frac{\log \mu(B(x,\delta))}{\log\delta},\quad
	\overline{\dim}(x) =\limsup_{\delta\to0}\frac{\log \mu(B(x,\delta))}{\log\delta},
\end{eqnarray*}
where $B(x,\delta)=\{y\in X: d(y,x)<\delta\}.$
The measure $\mu$ is said to have local dimension at $x$
if $\underline{\dim}(x)=\overline{\dim}(x).$ 
For a subset $Y\subset X$ and a number $t> 0$,  the $t$-Hausdorff measure of $Y$ is defined by
\[m_{H}(Y,t)=\lim_{\vep\to0}\inf_{\mathcal{U}} \sum\nolimits_{U\in \mathcal{U}} \diam^{t}(U),\]
where the infimum is taken over all finite or countable coverings  $\mathcal{U}$ of $Y$ by open sets with $\diam( \mathcal{U})\le \vep$. 
The Hausdorff dimension of $Y$ and then of $\mu$ are respectively defined as
\begin{eqnarray*} 
&&\dim_H (Y)=\inf\big{\{}t: m_H(Y, t) = 0\big{\}} = \sup\big{\{}t: m_H(Y, t) = \infty\big{\}},\\[2mm] 
&&\dim_H(\mu)=\inf\big{\{}\dim_H(Y): \mu(Y)=1\big{\}}.
\end{eqnarray*}

The following classical result says that under standard conditions, the local dimension exists and coincides with the Hausdorff dimension. 
\begin{pro}[Young \cite{Young82}]\label{prop:young} Let $X$ be a compact separable metric space of
	finite topological dimension and $\mu$ be a finite Borel measure on $X$ satisfying
	\begin{eqnarray}\label{exact}\underline{\dim}(x) =\overline{\dim}(x) = D,\quad\mu{\text{-a.e.}}\ x\in X.
	\end{eqnarray}
Then $\dim_H(\mu) =  D.$
\end{pro}

Any measure  $\mu$ satisfying \eqref{exact} is said to be {\it exact dimensional} with the common value denoted by $\dim(\mu).$

\section{upper semi-continuity of folding entropy }\label{sec:degenerate}
In this section, we prove  Theorem \ref{thm1}.  Throughout, let $f$ be a $C^r (r>1)$ map on a compact  Riemannian manifold $M.$ Before proceeding to the proof, we would like to note that the basic idea in the handling of the degeneracy in the proof,  though are based on that in \cite{LW} with some further refining estimations,  are more clearly explained and revealed here.

\subsection{Construction of refining partitions}
For convenience of analysis,  we embed $M$ into an  Euclidean space $\mathbb{R}^N.$ 
Then, one can find tubular neighborhoods  ${T_1}\subset  T_2$ of $M$ in $\mathbb{R}^N,$
 and a $C^{r}$ extension $g$ of $f$ from $ T_1$  to $T_2$   such that  $g^{-1}(\overline{T}_1))\subset  T_1$ and
$M=\cap_{n\in \mathbb{N}}\,g^{-n}(T_1)$.
Without confusion, we still write $g|_{T_1}$ as $f$ for simplicity. 
Let $\alpha=\min\{r-1,1\}.$ Then the derivative of $f$, denoted as $Df,$ is $\alpha$-H\"{o}lder continuous, i.e, there
exists $K>1$  such that
\begin{eqnarray}\label{Holder}
\|D_xf-D_yf\|\leq\,Kd(x,y)^{\alpha},\quad \forall\,x, y\in T_1.
\end{eqnarray}

To approximate the folding entropy, we shall use a sequence of {\it refining}  finite partitions to approximate the measurable partitions $\{\eps|_{f^{-1}\eps(x)}\}_{x\in T_1}$.  
To begin with, for each $k\in \N,$  define the partition $\Gamma_k$ of $\R^N$ as  \[\Gamma_k=\Big{\{}(\frac{q_1}{2^k},\frac{q_1+1}{2^k})
\overbrace{\times\cdots\times}^{N}(\frac{q_N}{2^k},\frac{q_N+1}{2^k})
~:~q_1,\cdots,q_N \in \mathbb{Z}\Big{\}}.\]
Plainly,  $\Gamma_1\preceq\cdots\preceq\Ga_k\preceq\cdots.$
Since $f$ only acts on $T_1,$  henceforth, by $\Gamma_k$ we mean $\Gamma_k|_{T_1}.$
Note that there exists a constant $C_1>0$, independent of $k, N$, such that
\begin{eqnarray}\label{number}
\#\big{\{}A\in \Gamma_k: A\cap T_1\neq \emptyset \big{\}} \leq C_12^{kN}, \quad \forall\,k\in \mathbb{N}.
\end{eqnarray}

Intuitively, to approximate the $f$-inverse partition $f^{-1}\eps,$
we need to consider the
``pull-back" of $\{\Gamma_k\}.$ 
In achieving this, we have to deal with the degeneracy of $f$. 
Given any $\vep>0$,  define
 \[U_{\vep}=\big{\{}x\in T_1: m(D_xf)<\vep\big{\}}, \quad G_{\vep}=\big{\{}x\in T_1:  m(D_xf)\geq\vep\big{\}},\]
 where $m(\Phi):=\inf_{\|z\|=1}\|\Phi(z)\|$ denotes the small norm of  a linear operator $\Phi$.  By  \eqref{Holder}, if take $r_{\vep}=(\frac{\vep^2}{4K})^{\frac{1}{\alpha}},$ we have  $m(D_xf)>\frac{\vep}{2}, \forall x\in G_{\vep}.$
 For each $k\in\N,$ let 
$\vep_k=\vep_02^{-k\beta}$ 
with $\beta=\frac{\alpha}{2+\alpha}$ and $ \vep_0=(2(4K)^{\frac{1}{\alpha}}\sqrt{N})^{\beta}.$ 
Then corresponding to each $\Ga_k$, $k\in\N,$ 
\begin{eqnarray}\label{contain}
f(B(x,r_{\vep_k}))\supseteq \Gamma_k(fx),\quad\forall x\in G_{\vep_k},
\end{eqnarray}
where  $\Gamma_k(fx)$ denotes the element of $\Gamma_k$ containing $fx.$
Obviously, $\{U_{\vep_k}\}$ is a sequence of decreasing neighborhoods of $\Sigma_f$ as $k\to\infty.$

Now,  for each $P\in \Gamma_k$,  denote $P^{-1, c}$ as the set of all the connected components of $f^{-1}P.$  
Separating  and collecting the components ``near" $\Sigma_f$,   let \[B_k=\bigcup_{P\in \Gamma_k}\,\, \bigcup_{\substack{Q\in P^{-1, c},\\ Q\cap U_{\vep_k}\neq \emptyset}}\,\, Q.\]
Obviously, $B_k\supset U_{\vep_k};$ see Figure \ref{fig:partition_def}(A). 
The following lemma shows that $B_k$ is contained in certain $U_{\delta}$ for $\delta$ being the same order as $\vep_k.$
\begin{lemma}\label{lem:B_k}
For any $k\ge1,$
 	\begin{eqnarray*}
 		B_k\subset U_{C_22^{-k\beta}},
 	\end{eqnarray*}	
for some constant $C_2>0$ independent of $k, \beta.$
 \end{lemma}
 \begin{proof}
We claim that for any $Q\in P^{-1, c},$ where $P\in\Ga_k,$ such that $Q\cap U_{\vep_k}\neq\varnothing,$ it holds that 
\begin{eqnarray}\label{non-empty}
B(x, r_{\vep_k})\cap U_{\vep_k}\neq \varnothing,\quad\forall x\in Q.
\end{eqnarray}
For otherwise, if for some $x_0\in Q$ such that $B(x_0, r_{\vep_k})\cap U_{\vep_k}= \varnothing,$ then 
by \eqref{contain}, $B(x_0, r_{\vep_k})$ contains certain preimage component of $f^{-1}\big(\Ga_k(fx_0)\big),$ say $Q_0,$ such that $x_0\in Q_0$. Hence, 
$Q=Q_0.$ This contradicts with  $Q\cap U_{\vep_k}\neq\varnothing.$

Now, together with  \eqref{Holder},  \eqref{non-empty} yields
\[m(D_xf)\le \vep_k+K r_{\vep_k}^{\alpha}\le C_22^{-k\beta},\quad\forall x\in B_k,\] 
for some constant $C_2>0.$ Hence, $B_k\subset U_{C_22^{-k\beta}}.$
 \end{proof}

Define the {\it pullback partition} of  $\Gamma_k,$ denoted by $\Gamma_k^{-1,c}$,  as \[\Gamma^{-1, c}_k:= \{B_k\}\bigcup_{P\in \Gamma_k} \{Q\in P^{-1, c}: Q\cap U_{\vep_k}= \emptyset\}.\]
We note that $\Ga^{-1,c}_k$ is a {\it finite approximation} of the  measurable partition $\eps|_{f^{-1}\eps}$ in such a way that all the preimages away from  $\Sigma_f$ are separated by different components $Q\in P^{-1,c},$ while those close to $\Sigma_f$ are collected by $B_k;$ see Figure \ref{fig:partition_def}(B). We call $B_k$ the {\it degenerate component} of $\Ga_k^{-1,c}.$
In the following, we shall use the conditional entropy
$H_\mu(\Ga_k^{-1,c}|f^{-1}\Ga_k)$ to 
approximate the folding entropy 
through the refining process $k\to\infty.$ 
\begin{figure}
\includegraphics[height=3.8cm]{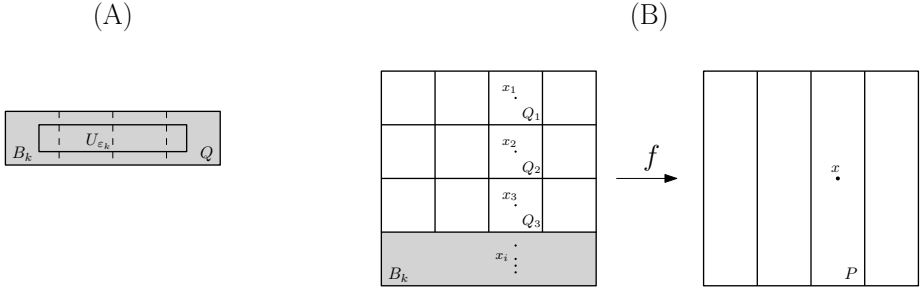}
\caption{(A) The degenerate component $B_k$ consists of all the  components $Q\in\Ga_k$ (denoted by dashed line) that intersect with $U_{\vep_k}.$ (B) The left panel denotes the pullback partition $\Ga_k^{-1,c};$  the right panel denotes the partition $\Ga_k.$ For a typical $P\in\Gamma_k,$  $x\in P,$ the preimages of $x,$ $x_i (i=1,2,3)$  lie in the different components $Q_i\in P^{-1,c} ( i=1,2,3).$ The other preimages  $x_i (i=4,\cdots)$ are all collected by the degenerate component $B_k\in\Ga_k^{-1,c}.$ }
	\label{fig:partition_def}
\end{figure}

\subsection{Approximating the folding  entropy}\,\smallskip
Up to the end of this section, 
we consider measures  $\nu\in\cP(M)$ for  certain degenerate rate $\bar\eta.$ Then it holds  that
\begin{eqnarray*}
\Big|\int \log|\Jac(D_xf)|d\nu(x)\Big|<\infty.
\end{eqnarray*}
Hence, $\nu(B_k)\to0$ as $k\to\infty.$
By construction,  for $\nu$-a.e. $x,$
$\Gamma^{-1, c}_k$ is increasing  to $\epsilon|_{(f^{-1}\eps)(x)}$ as $k\to\infty$. 
Therefore, 
\begin{eqnarray}\label{folding_estimat}
F_f(\nu) &=& \lim_{k\to +\infty}H_{\mu}(\Gamma^{-1, c}_k\mid f^{-1}\epsilon)\nonumber\\[2mm]
&\le&  \limsup_{k\to +\infty} H_{\mu}(\Gamma^{-1, c}_k\mid f^{-1}\Gamma_k)\nonumber\\[2mm]
&=&\limsup_{k\to\infty}\sum_{ P\in \Gamma_k} \sum_{Q\in\Gamma^{-1, c}_k}-\nu(Q\cap f^{-1}P) \log\Big{(}\frac{\nu(Q\cap f^{-1}P)}{\nu(f^{-1}P)}\Big{)}.\label{sum2}
\end{eqnarray}

For the (upper semi-)continuity analysis later, we shall split the summation in \eqref{folding_estimat} into two parts 
to collect the complexities near and away from $\Sigma_f,$ respectively. 
First, we note that for any $Q\in\Ga_k^{-1,c},$ either $Q=B_k,$ or $Q\cap B_k=\varnothing,$ 
where in the latter case, $f|_Q$ is a diffeomorphism such that 
$fQ=P$ for certain $P\in\Ga_k.$
Now according to whether $Q$ equal to $B_k$ or not, the RHS of \eqref{sum2} is split into two parts:
\begin{eqnarray*}
\Delta_k^{(1)}(\nu)&=&\sum_{ P\in \Gamma_k}-\nu(B_k\cap f^{-1}P) \log\Big{(}\dfrac{\nu(B_k\cap f^{-1}P)}{\nu(f^{-1}P)}\Big{)},\\
\Delta_k^{(2)}(\nu)&=&\sum_{\substack{Q\in\Gamma^{-1, c}_k\setminus \{B_k\},\\P=fQ}}-\nu(Q) \log\Big{(}\dfrac{\nu(Q)}{\nu(f^{-1}P)}\Big).
\end{eqnarray*}
Then alternatively,  we have 
\begin{eqnarray}\label{folding_ineq}
F_f(\nu)\le\limsup_{k\to\infty}\big(\Delta_k^{(1)}(\nu)+\Delta_k^{(2)}(\nu)\big).
\end{eqnarray}
In the following, we will analyse $\Delta_k^{(1)}$ and $\Delta_k^{(2)},$ respectively. 

\subsubsection{Estimation of $\Delta_k^{(1)}$}
This subsection shows that  as $k$ increases, the complexities collected by the degenerate component $B_k,$  $\Delta_k^{(1)},$ is uniformly small 
on any $\cP_{\bar\eta}(f).$
First, we present a technical lemma which establishes the relationship between  the degeneracy of maps and the   degeneracy of   measures.

\begin{lemma}\label{lem:technical}
For any $k\ge1$,  it holds that
\begin{eqnarray*}
k\nu(B_k)\le C_3\nu(B_k)-C_4\int_{B_k}\log|\Jac(D_xf)|d\nu(x),
\end{eqnarray*}
where the constants $C_3,C_4>0$ are independent of $k$ and $\nu.$
\end{lemma}
\begin{proof}
By Lemma \ref{lem:B_k}, for any $x\in B_k,$  
\begin{eqnarray*}
		|\Jac(D_xf)|\le C_2L^{N-1}2^{-\beta k},
	\end{eqnarray*}
where $L:=\max_{x\in T_1}\|D_xf\|.$ This can be equivalently written as
\begin{eqnarray*}\label{k}
k\le -\dfrac{1}{\beta\log2}\log\big(|\Jac(D_xf)|/C_2L^{N-1}\big),\quad\forall x\in B_k.
\end{eqnarray*}
Therefore, 
\begin{eqnarray*}
\nu(B_k)k &\leq& -\frac{1}{\beta\log2}\int_{B_k}\log\big(|\Jac(D_xf)| /(C_2L^{N-1})\big)d\nu(x)\\[2mm]&=& \frac{\nu(B_k)\log(C_2L^{N-1})}{\beta\log2}-\frac{1}{\beta\log2}\int_{B_k}\log  |\Jac(D_xf)| d\nu(x).
\end{eqnarray*}
The Lemma is proved by taking $C_3=\frac{\log(C_2L^{N-1})}{\beta\log2}, C_4=\frac{1}{\beta\log2}.$
\end{proof}

The following simple fact, as a consequence of Lemma \ref{lem:technical},   is helpful throughout this section. 
\begin{lemma}\label{lem:uniform_measure} 
For any $\nu\in\cP_{\bar\eta}(f),$ 
$\nu(B_k)\to0$ uniformly as $k\to\infty.$
\end{lemma}
\begin{proof}
By Lemma \ref{lem:technical},  we only need to notice that for any $k$ sufficiently large,  the term 
$$\big{|} \int_{B_k}\log|\Jac(D_xf)|d\nu(x)\big{|}$$  
is uniformly small  for  $\nu\in\cP_{\bar\eta}(f)$.  
\end{proof}
	
\medskip	
	
Now, we are prepared to estimate $\Delta_k^{(1)}$. 

\begin{lemma}\label{lem:1,1} For any $\vep>0$,  there exists $k_1\in \mathbb{N}$ such that for any  $\nu\in \mathcal P_{\bar \eta}(f)$ and $k\ge k_1$,   it holds that \[ \Delta_k^{(1)}(\nu) <\vep.\]
\end{lemma}
\begin{proof} 
First,  by noting that $\log\nu(f^{-1}P)\le0,$ we have
 \begin{eqnarray}
	\Delta_k^{(1)}(\nu)	
	&\le&\sum_{ P\in \Gamma_k}-\nu(B_k\cap f^{-1}P) \log\nu(B_k\cap f^{-1}P)\nonumber\\
	&=&\nu(B_k)\sum_{ P\in \Gamma_k}\phi\Big{(}\dfrac{\nu(B_k\cap f^{-1}P)}{\nu(B_k)}\Big{)}-\nu(B_k)\log\nu(B_k).\label{ineq}
\end{eqnarray}
Applying Proposition \ref{prop:phi} with $\{p_i\}=\Big\{\dfrac{\nu(B_k\cap f^{-1}P)}{\nu(B_k)}\Big\}_{P\in\Ga_k},$ 
we have
\begin{eqnarray*}&&\sum_{ P\in \Gamma_k}\phi\Big{(}\dfrac{\nu(B_k\cap f^{-1}P)}{\nu(B_k)}\Big{)}\le\log(\#\Ga_k)\le\log (C_12^{kN}),
\end{eqnarray*}
where the last inequality comes form \eqref{number}.
Hence, it holds that 
\begin{eqnarray*}
\Delta_k^{(1)}(\nu)\le \log(C_1^{1/k}2^N)k\nu(B_k)
-\nu(B_k)\log\nu(B_k),
\end{eqnarray*}
which,  by Lemma \ref{lem:technical},  yields
\begin{eqnarray*}
\Delta_k^{(1)}(\nu)\le \tilde C_3\nu(B_k)-\tilde C_4\int_{B_k}\log|\Jac(D_xf)|d\nu(x)-\nu(B_k)\log\nu(B_k),
\end{eqnarray*}
where the constants $\tilde C_3,\tilde C_4>0$ are independent of $k,\nu.$ 
Since the convergence \[\int_{B_k}\log  |\Jac(D_xf)| d\nu(x)\to0,\quad k\to\infty\]
is uniform  for any $\nu\in \mathcal P_{\bar \eta}(f)$. Then combined with Lemma \ref{lem:uniform_measure}, 
the lemma is proved. 
\end{proof}

\subsubsection{Analysis of $\Delta_k^{(2)}$}\label{subsec:Delta_k^2}
For analyzing $\Delta_k^{(2)}$, it is useful to carefully  investigate 
the partition refining process.  Note that as $i$ increased from $(k-1)$ to $k$, for the partition $\Ga_i^{-1,c},$
two scenarios occur simultaneously: the partition outside $B_{k-1}$ becomes finer; meanwhile, new connected components emerge, which are integrated as a whole by $B_{k-1}$,   from $(B_{k-1}\backslash B_{k})$. 
By such observation, for each $k\ge1$
define 
\begin{eqnarray*}
\Gamma^{-1, c}_{k,1}&=&\Big{\{}Q\in	\Gamma^{-1, c}_{k}:  Q\,\,\text{is contained in some}\,\, Q'\in 	\Gamma^{-1, c}_{k-1}\setminus \{B_{k-1}\}\Big{\}},\\
\Om_{k}&=& \Big{\{}Q\in\Gamma^{-1,c}_{k}\setminus\{B_{k}\}: Q\text{ is contained in } B_{k-1}\Big{\}},
\end{eqnarray*}
see Figure \ref{fig:finer_partition}(A).
\begin{figure}
\includegraphics[height=3.5cm]{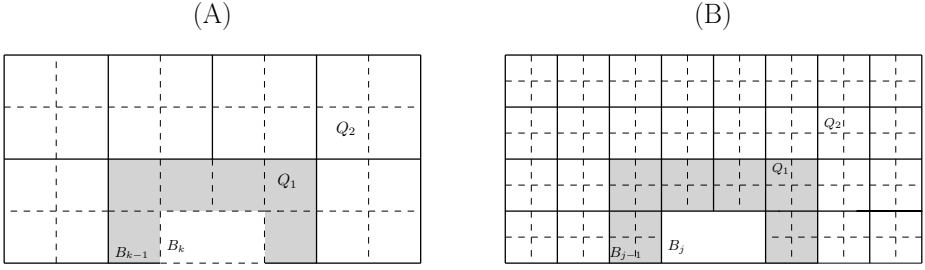}
\caption{(A) The solid line denotes the partition $\Ga_{k-1}^{-1,c};$ the dashed line denotes the refined partition $\Ga_k^{-1,c}.$	The set difference $B_{k-1}\backslash B_k$ (the region in grey) is consists of the new-emerging components $Q_1\in\Om_k.$ Outside $B_{k-1},$ the partition $\Ga_{k-1}^{-1,c}$ is refined by the components $Q_2\in\Gamma_{k,1}^{-1,c}.$
(B) The solid line denotes the partition $\Ga_{j}^{-1,c};$ the dashed line denotes the refined partition $\Ga_k^{-1,c}.$	The set difference $B_{j-1}\backslash B_j$ (the region in grey) is refined by components $Q_1\in\Om'_{k,j}.$ Outside $B_{j-1},$ the partition $\Ga_{j}^{-1,c}$ is refined by  components $Q_2\in\cW_{k,j}.$}
\label{fig:finer_partition}
\end{figure}
Correspondingly, we also split  $\Delta_k^{(2)}$
into two parts: (i)
the {\it refining} complexities outside  $B_{k-1}$, denoted by $\Delta_{k}^{(2,1)}$; (ii) the {\it new} complexities released from  $(B_{k-1}\backslash B_{k}),$ denoted by $\Delta_{k}^{(2,2)},$ as follows:
\begin{eqnarray}
\Delta_{k}^{(2)}(\nu)
&=&\sum_{\substack{Q\in\Gamma^{-1, c}_{k,1},\\ P=fQ }}-\nu(Q) \log\Big{(}\dfrac{\nu(Q)}{\nu(f^{-1}P)}\Big)
+\sum_{\substack{Q\in\Om_{k},\\P=fQ}}-\nu(Q) \log\Big{(}\dfrac{\nu(Q)}{\nu(f^{-1}P)}\Big)\nonumber\\
&=:& \Delta_{k}^{(2,1)}(\nu)+\Delta_{k}^{(2,2)}(\nu).\label{split(2)}
\end{eqnarray}

The remainder of this subsection estimates 
$\Delta_{k}^{(2,2)}$ on any $
\cP_{\bar\eta}(f).$
In  the next subsection,  we will establish the  (non-increasing) monotonicity of $\Delta_{k}^{(2,1)}.$

\begin{lemma}\label{lem:(2,2)}
For any $\vep>0$,  there exists $k_2\in \mathbb{N}$ such that 	for any $k\ge k_2$ and  $\nu\in \cP_{\bar \eta}(f)$,  
	\begin{eqnarray*}
		\sum_{k>k_2}\Delta_{k}^{(2,2)}(\nu)< \vep.
	\end{eqnarray*}
\end{lemma}
\begin{proof}  
Since for any $k\ge1$, 
$\cup_{Q\in \Omega_{k}} Q=B_{k-1}\setminus B_{k}$, 
\begin{eqnarray}
	\Delta_k^{(2,2)}(\nu)
	&\le&\sum_{Q\in\Om_k}-\nu(Q)\log\nu(Q)\label{sum1}.
\end{eqnarray}
Similar to \eqref{ineq}, the RHS of \eqref{sum1} can be written as  
\begin{eqnarray*}
	&&\nu(B_{k-1}\setminus B_{k})\sum_{Q\in\Omega_{k}} \phi\Big{(}\frac{\nu(Q)}{\nu(B_{k-1}\setminus B_{k})}\Big{)}-\nu(B_{k-1}\setminus B_{k})\log\nu(B_{k-1}\setminus B_{k}) \\[2mm]
	&=:& \Delta_{k}^{(2,2,1)}(\nu)+\Delta_{k}^{(2,2,2)}(\nu).
\end{eqnarray*} 	
Thus, the estimation of $\Delta_k^{(2,2)}$ is reduced to that  of $\Delta_{k}^{(2,2,1)}$ and $\Delta_{k}^{(2,2,2)},$ respectively. 
	
We use  similar approach 
in Lemma \ref{lem:1,1} to estimate  $\Delta_{k}^{(2,2,1)}.$ 
First, note that each $Q\in\Omega_{k}$ contains a ball with radius at least
$(2^{k+1}L)^{-1}.$ Thus, $\#\Omega_{k}\le \tilde C_1^{N(k+1)}L^N$ 
for some constant $\tilde C_1>0$ independent of $k, L$ and $N$.
Then applying Proposition \ref{prop:phi} with $\{p_i\}=\{\frac{\nu(Q)}{\nu(B_{k-1}\setminus B_{k})}\}_{Q\in\Om_k},$  we have 
\begin{eqnarray}
\Delta_{k}^{(2,2,1)}(\nu)&\le&\log(\#\tilde C_1^{N(k+1)}L^N),
\end{eqnarray}
which, using the estimatition as in  Lemma \ref{lem:technical} with  $B_k$ replaced by $(B_{k}\backslash B_{k-1}),$ yields
\begin{eqnarray*}
\Delta_{k}^{(2,2,1)}(\nu)&\le&\tilde C_3\nu(B_{k-1}\setminus B_{k})-\tilde C_4\int_{B_{k-1}\setminus B_{k}}\log  |\Jac(D_xf)| d\nu(x)
\end{eqnarray*}
where the constants $\tilde C_3,\tilde C_4>0$ are independent of $k$ and $\nu.$
Thus, for any $k_2\in\N,$
\begin{eqnarray}\label{221}
\sum_{k>k_2}\Delta_{k}^{(2,2,1)}(\nu)&\le& \tilde C_3\nu(B_{k_2})-\tilde C_4\int_{B_{k_2}}\log  |\Jac(D_xf)| d\nu(x).
\end{eqnarray}

\medskip

For the estimation of $\Delta_{k}^{(2,2,2)},$ denote 
\begin{eqnarray*}I_1=\big{\{}k\in\N:\,\, k\ge -\log\nu(B_{k}\setminus B_{k-1})\big{\}},\quad 
I_2=\big{\{}k\in\N:\,\, k<-\log\nu(B_{k}\setminus B_{k-1})\big{\}}.
\end{eqnarray*}
Then for any $k_2\in\N,$
\begin{eqnarray*}
\sum_{k>k_2}\Delta_k^{(2,2,2)}&=& \sum_{k\in I_1, k>k_2} -\nu(B_{k}\setminus B_{k-1})\log\nu(B_{k}\setminus B_{k-1})\\[2mm]
&&+
\sum_{k\in I_2, k>k_2} -\nu(B_{k}\setminus B_{k-1})\log\nu(B_{k}\setminus B_{k-1})\\[2mm]
&\le&\sum_{k\in I_1, k>k_2} k\nu(B_{k}\setminus B_{k-1})+\sum_{k\in I_2 , k>k_2} ie^{-i}.
\end{eqnarray*}
Furthermore, applying Lemma \ref{lem:technical}, we have  
\begin{eqnarray}\label{222}
\sum_{k>k_2}\Delta_k^{(2,2,2)}&\le&C_3\nu(B_{k_2})-C_4\int_{B_{k_2}}\log  |\Jac(D_xf)| d\nu(x)+\sum_{k\in I_2 , k>k_2} ke^{-k}. 
\end{eqnarray}

Note that  $\sum_{k>k_2} ke^{-k}\to 0$ as $k_2\to\infty$. Also, as $k_2\to\infty,$
both $\int_{B_{k_2}}\log|\Jac(D_xf)|d\nu(x)$ and $\nu(B_{k_2})$ converge to zero
uniformly on $\cP_{\bar\eta}(f)$.
Then combing \eqref{221} and \eqref{222}, the lemma is proved.
\end{proof}

\subsubsection{Monotonicity by refining partitions}
In this subsection, we demonstrate the (non-increasing) monotonicity for complexities outside any  degenerate component $B_{k}\in\Ga_k^{-1,c}$
during the partition refining process. We remark that distinct from the previous two subsections where the uniform estimations  
require the measures to  
admit a uniform degenerate rate,  the monotonicity analysis in this subsection applies to all measures in $\cP(M).$

For any $1\le j\le k$,  define
\begin{eqnarray*}
\cW_{k,j}&=&\Big\{Q\in \Gamma^{-1,c}_{k}\backslash \{B_{k}\}: Q\subseteq Q' \text{\ for\ some \ } Q'\in\Gamma^{-1,c}_{j}\backslash\{B_{j}\}\Big\},\\
\Omega'_{k,j}&=&\Big{\{}Q\in\Gamma^{-1, c}_{k}\setminus\{B_{k}\}: Q \subseteq Q'\,\,\text{for some}\,Q'\in \Omega_{j} \Big{\}}.
\end{eqnarray*}
In the same spirit as  $\Gamma_{k,1}^{-1,c}$ and $\Om_k$ in section \ref{subsec:Delta_k^2},  both $\cW_{k,j}$ 
and $\Ga_{k,1}^{-1,c}$ 
are to refine the elements in 
$\Ga_j^{-1,c}\backslash\{B_j\}$
by using 
the more refined partition $\Ga_k^{-1,c}$;  see Figure \ref{fig:finer_partition}(B).
In particular, $\cW_{k,k-1}=\Gamma_{k,1}^{-1,c},$ $\Om'_{k,k}=\Om_k.$
We denote the respective complexities captured by $\cW_{k,j}$ and $\Omega'_{k,j}$ as
\begin{eqnarray*}
I_{k,j}(\nu)=\sum_{\substack{Q\in\cW_{k,j},\\P=fQ}}-\nu(Q) \log\Big{(}\dfrac{\nu(Q)}{\nu(f^{-1}P)}\Big),\quad
I'_{k,j}(\nu)=\sum_{\substack{Q\in\Om'_{k,j},\\P=fQ}}-\nu(Q) \log\Big{(}\dfrac{\nu(Q)}{\nu(f^{-1}P)}\Big).
\end{eqnarray*}
As a natural linkage to $\Delta_{k}^{(2)}$,
it is not hard to see the following. 
\begin{pro}\label{prop:complex_split}
Fix any $k_0\in\N.$ Then for any $k>k_0$ and  $\nu\in\cP(M),$
\begin{eqnarray*}
\Delta_k^{(2)}(\nu)=I_{k,k_0}(\nu)+\sum_{j=k_0+1}^{k}I'_{k,j}(\nu).
\end{eqnarray*}
In particular, \eqref{split(2)} is yielded by taking $k_0=k-1,$ i.e., 
\begin{eqnarray*}
\Delta_k^{(2)}(\nu)=I_{k,k-1}(\nu)+I'_{k,k}(\nu),
\end{eqnarray*}
where $I_{k,k-1}=\Delta_k^{(2,1)}, I'_{k,k}=\Delta_k^{(2,2)}.$
\end{pro}

\begin{lemma}\label{lem:monotone}
Given any $\nu\in\cP(M)$ and $j\in\N.$ Then for any $k\ge j,$ $I_{k,j}(\nu)$ ({\it resp.} $I'_{k,j}(\nu)$) is non-increasing with respect to $k.$ In particular, $I'_{k,j}(\nu)\le I'_{j,j}(\nu)=\Delta_{j}^{(2,2)}(\nu).$
\end{lemma}
\begin{proof}
We only  prove the non-increasing monotonicity for $I_{k,j}.$ The same argument applies to $I'_{k,j}$ as well. 
A key observation is that for any $k'>k\ge j,$
the partitions $\Ga_{k'}^{-1,c}$ and $\Ga_{k'}$  refine the respective $\Ga_{k}^{-1,c}$ and $\Ga_{k}$  in a consistent manner.  More specifically, 
for any $Q\in\cW_{k,j},$ 
since $f|_Q$ is a diffeomorphism, 
then there is a one-one correspondence, through $f,$  between the refining of $P$ by $\Ga_{k'}$ and the refining of $Q$ by $\cW_{k',j}$; see Figure \ref{fig:mono_refin}.
\begin{figure}
\includegraphics[height=3.5cm]{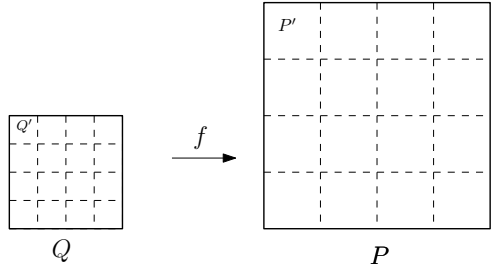}
\caption{The component $Q\in\cW_{k,j}$ ({\it resp.} $P\in\Ga_k$) is refined by components $Q'\in\cW_{k',j}$ ({\it resp.} $P'\in\Ga_{k'}$) in a one-one corresponding way such that each $P'=fQ'.$}
\label{fig:mono_refin}
\end{figure}
Thus,  $I_{k',j}$ can be written as
\begin{eqnarray}
I_{k',j}&=&
\sum_{Q\in\cW_{k,j}}\sum_{\substack{Q'\in\cW_{k',j},\\Q'\subseteq Q, P'=fQ'}}\nu(f^{-1}P')\phi\Big(\dfrac{\nu(Q')}{\nu(f^{-1}P')}\Big)\nonumber\\
&=&\sum_{\substack{Q\in\cW_{k,j},\\P=fQ}}\nu(f^{-1}P)\sum_{\substack{Q'\in\cW_{k',j},\\ Q'\subseteq Q, P'=fQ'}}\dfrac{\nu(f^{-1}P')}{\nu(f^{-1}P)}\phi\Big(\dfrac{\nu(Q')}{\nu(f^{-1}P')}\Big)\label{sum}.\label{sum3}
\end{eqnarray}

Since $\{P'\}$  is a refinement of $P$, applying Proposition \ref{prop:phi}  with $\{p_i\}=\{\frac{\nu(f^{-1}P')}{\nu(f^{-1}P)}\}_{P'},$ $\{x_i\}=\{\frac{\nu(Q')}{\nu(f^{-1}P')}\}_{Q'},$ for the first summation in \eqref{sum3}, we have
\begin{eqnarray*}
\sum_{\substack{Q'\in\cW_{k',j},\\ Q'\subseteq Q, P'=fQ'}}\dfrac{\nu(f^{-1}P')}{\nu(f^{-1}P)}\phi\Big(\dfrac{\nu(Q')}{\nu(f^{-1}P')}\Big)\le \phi\Big(\sum_{\substack{Q'\in\cW_{k',j},\\ Q'\subseteq Q,}}\dfrac{\nu(Q')}{\nu(f^{-1}P)}\Big)=\phi\Big(\dfrac{\nu(Q)}{\nu(f^{-1}P)}\Big). 
\end{eqnarray*}
Therefore, 
\begin{eqnarray*}
I_{k',j}\le\sum_{\substack{Q\in\cW_{k,j},\\P=fQ}}\nu(f^{-1}P)\phi\Big(\dfrac{\nu(Q)}{\nu(f^{-1}P)}\Big)=\sum_{\substack{Q\in\cW_{k,j},\\P=fQ}}-\nu(f^{-1}P)\log\Big(\dfrac{\nu(Q)}{\nu(f^{-1}P)}\Big)=I_{k,j}. 
\end{eqnarray*}
\end{proof}

Fixing  any  $k_0\in\N,$ by Lemma \ref{lem:monotone},  it is well-defined to put
\begin{eqnarray*}\label{limit_I}
I_{k_0}(\nu)=\lim\limits_{k\rightarrow\infty}I_{k,k_0}(\nu)=\inf\limits_{k\geq k_0}I_{k,k_0}(\nu).
\end{eqnarray*}
For each $k\ge k_0,$ since $\cW_{k,k_0}$  only  contains the components of $\Ga_{k}^{-1,c}$ not intersecting with the degenerate component $B_{k_0}\in\Ga_{k_0}^{-1,c}$.
Then as $k$ increases, $I_{k,k_0}$ approximates only the part of complexities in folding entropy  away from $\Sigma_f$.
Therefore, 
\begin{eqnarray}\label{folding inequality}
I_{k_0}(\nu)\le F_f(\nu),\quad\forall \nu\in\cP(M).
\end{eqnarray}
Also, as each $I_{k,k_0}$ is a finite summation, we directly obtain the  upper semi-continuity property  of $I_{k_0}$ on $\cP(M).$ 

 \begin{pro}\label{prop: u.s.c I}
For each $k\ge1,$ $I_k(\cdot)$ is upper semi-continuous on $\cP(M).$
\end{pro}
\begin{proof}
Given any $\mu\in\cP(M).$	By a translation if necessary, we may assume $\mu(\partial \Gamma_k)= \mu(\partial\Ga_k^{-1,c})=0$ for all $k\in\N.$
Thus, for any $k\ge k_0,$
\[I_{k,k_0}(\nu)=\sum_{\substack{Q\in\cW_{k,k_0},\\P=fQ}}-\nu(Q) \log\Big{(}\dfrac{\nu(Q)}{\nu(f^{-1}P)}\Big)\]	is continuous at $\mu.$ This implies the upper semi-continuity of $I_k=\inf_{k\ge k_0}I_{k,k_0}$ on $\cP(M).$
\end{proof}

\bigskip

Now we are at the position to prove Theorem \ref{thm1}.

\medskip

\noindent{\it Proof of Theorem \ref{thm1}.}
Given any $\vep>0,$ by  Lemmas \ref{lem:1,1} and  \ref{lem:(2,2)}, for $k_0=\min\{k_1,k_2\}$ we have  
\begin{eqnarray*}
\Delta_k^{(1)}(\nu)<\vep,\quad \sum_{k_0\le j\le k}\Delta_{j}^{(2,2)}(\nu)
<\vep,\quad\forall k\ge k_0.
\end{eqnarray*}
By Proposition \ref{prop:complex_split} and Lemma \ref{lem:monotone},
\begin{eqnarray*}
\Delta_{k}^{(2)}(\nu)
\le I_{k,k_0}(\nu)+\sum_{k_0<j\le k}\Delta_{j}^{(2,2)}(\nu)
<I_{k,k_0}(\nu)+\vep.
\end{eqnarray*}
It then follows from \eqref{folding_ineq} that \begin{eqnarray}\label{estimation}
F_f(\nu)<I_{k,k_0}(\nu)+2\vep,\quad\forall k\ge k_0, 
\end{eqnarray}
where we enlarge $k_0$ if necessary.

Now, let $\{\mu_i\}\subseteq\mathcal M_{\bar \eta}(f) $ such that   $\mu_i$ converge to $\mu$.
Applying \eqref{estimation} with $\nu$ being $\{\mu_i\}$,  together with Proposition \ref{prop: u.s.c I}, we have
\[\limsup_{i\to\infty} F_f(\mu_i)\le \limsup_{i\to\infty} I_{k,k_0}(\mu_i) +2\vep\le   I_{k,k_0}(\mu)+2\vep,\quad\forall k\ge k_0
,\]
which, by \eqref{folding inequality},
implies 
\[\limsup\limits_{i\to\infty} F_f(\mu_i)\le F_f(\mu)+2\vep.\]
The proof of Theorem \ref{thm1} is concluded by the arbitrariness of $\vep$.
\hfill $\Box$ 

\section{Applications to interval maps}\label{Application}

In this section, 
we focus on the one-dimensional setting
by considering $C^r(r>1)$ interval maps.
As applications of Theorem \ref{thm1}, we will show the upper semi-continuity of both the metric entropy $h_\mu(f)$ and the dimension $\dim(\mu)$ when  measures  with uniform degenerate rate are considered. 
To achieve this, we will establish a entropy formula and a dimension formula  to the very general nature of $f.$

\subsection{Upper semi-continuity of metric entropy}\,\smallskip
For a $C^{r}(r>1)$ interval map $f$ and $\mu\in\cM_{inv}(f),$ recall the following inequalities for the metric entropy $h_\mu(f)$:
\begin{eqnarray}
h_\mu(f)&\leq& \displaystyle\int\max\{\lambda(f,x), 0\} d\mu(x),\label{Ruelle_ineq_1D}\\
h_\mu(f)&\leq& 
F_f(\mu)-\displaystyle\int\min\{\lambda(f,x), 0\} d\mu(x),\label{Folding_Ruelle_ineq_1D}
\end{eqnarray}
where $\lambda(f,x)$ is  the Lyapunov exponent  of $f$ at $x$.    
The first inequality in \eqref{Ruelle_ineq_1D} is the well-known Margulis-Ruelle inequality \cite{Ruelle78}, and \eqref{Folding_Ruelle_ineq_1D}  is the  folding-type Ruelle inequality \eqref{folding_Ruelle_ineq} applied in the one-dimensional setting.
The following theorem shows that in the one-dimensional setting, the metric entropy and folding entropy are actually equal. 

\begin{theorem}\label{thm:entropy equality}  Let $f$ be a $C^{r} (r>1)$ map  on an interval  $I$.  Then for any  $\mu\in\mathcal M_{inv}(f)$,  it holds that 
\begin{eqnarray}\label{entropy_formula}
h_\mu(f)=F_f(\mu).
\end{eqnarray}
\end{theorem}

\begin{proof}
First, we have  that 
\begin{equation}\label{folding small} h_{\mu}(f) \ge F_\mu(f),\quad\forall\mu\in\cM_{inv}(f).
\end{equation}
To see this, let  $\{\xi_n\}_{n\ge 1}$ be a  sequence of increasing finite partitions of $I$ such that $\bigvee\limits_{n\geq1}\xi_n=\epsilon (\mod \mu)$. 
Then by \eqref{metric-entropy-by-conditional-entropy}, for each $n\ge1,$
\begin{eqnarray*}
h_\mu(f,\xi_n)=H_\mu(\xi_n|\bigvee\limits_{i=1}^{\infty}f^{-i}\xi_n)\geq H_\mu(\xi_n|f^{-1}\epsilon),
\end{eqnarray*}
where the inequality is by observing that 
$f^{-1}\epsilon\succeq\bigvee \limits_{i=1}^{\infty}f^{-i}\xi_n=f^{-1}(\bigvee\limits_{i=0}^{\infty}f^{-i}\xi_n).$ 
Thus, 
\begin{eqnarray*}
h_\mu(f)=\lim\limits_{n\rightarrow\infty}h_\mu(f,\xi_n)\geq \lim\limits_{n\rightarrow\infty}H_\mu(\xi_n|f^{-1}\epsilon)=H_\mu(\epsilon|f^{-1}\epsilon).
\end{eqnarray*}  

\medskip

Now, we only need to show $h_{\mu}(f) \le F_f(\mu).$ Denote
\begin{eqnarray*}
A_1=\{x\in I: \lambda(f, x)\le 0\} \quad \text{and}\quad A_2=\{x\in I: \lambda(f, x)> 0\}.
\end{eqnarray*}
Without loss of generality, assume that $A_1$ and $A_2$ are both  positive $\mu$-measured, since for otherwise, the situations would be easier.  Let $\mu^{(i)}=\mu/ \mu(A_i), i=1,2$, and we write  $\mu=\sum_{i=1,2}\mu(A_i)\mu^{(i)}.$
	
By \eqref{Ruelle_ineq_1D} we have $h_{\mu^{(1)}}(f)=0.$ Then combined  with (\ref{folding small}), one  gets  
\begin{eqnarray}\label{folding1}
F_f(\mu)=h_{\mu^{(1)}}(f)=0.
\end{eqnarray}
For $\mu^{(2)},$ one directly gets, from  \eqref{Folding_Ruelle_ineq_1D}, that
\begin{eqnarray}\label{folding2}
h_{\mu^{(2)}}(f)\leq F_f(\mu^{(2)}).
\end{eqnarray}
Since both $A_1$ and $A_2$ are $f$-invariant (mod $\mu$), 
we have  
\begin{eqnarray*}
F_f(\mu)=\int_IH_{\tilde{\mu}_x}(\epsilon)d\mu
=\sum_{i=1,2}\int_{A_i} H_{\tilde{\mu}_x}(\epsilon)d\mu
=\sum_{i=1,2}\mu(A_i)F_f(\mu^{(i)}),
\end{eqnarray*}
where recall that $\tilde{\mu}_x$ is the disintegration  of $\mu$ along the partition $\{f^{-1}x\}.$
Now, combining  \eqref{folding1} and \eqref{folding2}, we have
\begin{eqnarray*}
h_\mu(f)=\sum_{i=1,2}\mu^{(i)}(A_i)h_{\mu^{(i)}}(f)\le \sum_{i=1,2}\mu^{(i)}(A_i) F_f({\mu^{(i)}})=F_f(\mu).
\end{eqnarray*}
\end{proof}

\noindent{\it Proof of Theorem \ref{thm2}.} 
By \eqref{entropy_formula}, Theorem \ref{thm2} is a direct consequence of Theorem \ref{thm1}.

\subsection{Upper semi-continuity of dimension}\,\smallskip
For an ergodic $\mu\in\cM_{inv}(f),$ the Lyapunov exponents of $f$  are  constants with respect to $\mu$-a.e. $x$,   and in the one-dimensional setting particularly, we have  only one  Lyapunov exponent and denote it as $\lambda(\mu)$.
The main estimation of this subsection is to establish the following {\it dimension formula} which relates the dimension to  metric entropy through  Lyapunov exponent for all $C^r(r>1)$ interval maps. 

\begin{theorem} \label{thm:dim_formula}
Let $f$ be a $C^{r}(r>1)$ map on an interval $I.$  Given a hyperbolic  $\mu\in\cM_{erg}(f)$,  
then  $\mu$ is exact dimensional and satisfies 
\begin{eqnarray}\label{dim_formula}
h_{\mu}(f)=\lambda^+(\mu) \dim(\mu), 
\end{eqnarray}
where $\lambda^+(\mu)=\max\{\lambda(\mu),0\}.$
\end{theorem}

\medskip

Before proceeding to its proof,  we would like to state two remarks of Theorem \ref{thm:dim_formula}:
(i) The hyperbolicity assumption of $\mu$ is sharp.  Both $C^r (r<\infty)$  and  analytical   examples are constructed in respective \cite{LM}  and \cite{KS}  showing that for a non-hyperbolic ergodic measure $\mu$ with zero exponent,  the  local dimension may not exist almost everywhere;
(ii) The formula \eqref{dim_formula} is proved by Ledrappier in \cite{Led} for an interval map $f$ under the assumption that the entropy of the partition by the degenerate sets of $f$ and $f'$ 
are  finite. Thus, Theorem \ref{thm:dim_formula} is a general result in this respect.

\medskip

A key step in the proof of Theorem \ref{thm:dim_formula} is   an  {\it integrable  version} of the classical {\it Brin-Katok formula}; see Lemma \ref{lem:int BK}.
Let $f$ be a continuous map on  a compact metric space $M$. For any $\delta>0$,   $n\in\mathbb{N}$ and $x\in M$, define
\[B_n(x,\delta)=\big\{y\in M:  d(f^ix, f^iy)<\delta, 0\le i\le n-1\big\}.\]
The following local entropy formula was established by Brin and Katok in \cite{BK}. 

\begin{pro}	[Brin-Katok  \cite{BK}]\label{prop:BK} 
Let  $\mu\in\cM_{inv}(f).$ Then for $\mu$-a.e. $x\in M$, 
\begin{eqnarray}\label{BK_formula}
\lim_{\delta\to0}\liminf_{n\to\infty} -\frac{\log \mu(B_n(x,\delta))}{n}=\lim_{\delta\to0}\limsup_{n\to\infty} -\frac{\log \mu( B_n(x,\delta))}{n}:=h_{\mu}(f,x),
\end{eqnarray}
where  the local entropy  $h_\mu(f,x)$ satisfies 
\begin{eqnarray*}
\int_X h_{\mu}(f,x)d\mu=h_{\mu}(f).
\end{eqnarray*}
In particular, if $\mu$ is ergodic, then \[h_{\mu}(f,x)=h_{\mu}(f),\quad \mu{\text{-a.e.}}\ x\in M.\] 
\end{pro}

\medskip

We will prove an integrable version of \eqref{BK_formula} 
by replacing $\delta$ 
with any  $\log$-integrable function. To be specific, 
let $\mathscr S$ be the set of all 
functions $\psi: M \to \mathbb{R}^+$ satisfying 
\begin{eqnarray*}
	\int -\log \psi(x)d\mu(x)<\infty.
\end{eqnarray*}
 For any $\psi\in\mathscr S$, $n\in \mathbb{N}$, and $x\in M$  define 
\[B_n(x,\psi)=\big\{y\in M: 
d(f^ix, f^iy)<\psi(f^ix), 0\le i\le n-1\big\}.\]
 It has been shown by Ma\~n\'e that 
\begin{pro}[Ma\~n\'e\cite{Mane}]\label{prop:mane}  For any $\psi\in \mathscr S$, 
\begin{eqnarray*}
\int\limsup_{n\to\infty} -\frac{\log \mu( B_n(x,\psi))}{n}d\mu(x)\le h_{\mu}(f).
\end{eqnarray*}
\end{pro}

Based on Propositions \ref{prop:BK} and  \ref{prop:mane}, we obtain the following generalized local entropy formula.

\begin{lemma}\label{lem:int BK}
Let $f$ be a continuous map on a compact metric space $M$ and $\mu\in\cM_{inv}(f).$ 
Then given any $\psi\in\mathscr{S},$ for $\mu$-a.e. $x\in M$,
\[\lim_{\delta\to0}\liminf_{n\to\infty}\ -\frac{\log \mu(B_n(x,\psi^{\delta}))}{n}=\lim_{\delta\to0}\limsup_{n\to\infty} -\frac{\log \mu( B_n(x,\psi^{\delta}))}{n}=h_{\mu}(f,x),\]
where $\psi^{\delta}(x):=\min\{\psi(x), \delta\}$ and $h_{\mu}(f,x)$ is given as in Proposition \ref{prop:BK}. 
\end{lemma}
\begin{proof} 
For one thing, since $B_n(x,\psi^\delta)\subseteq B_n(x,\delta)$, we have
\begin{eqnarray}\label{ineq1}
\lim_{\delta\to0} \liminf_{n\to\infty} -\frac{\log \mu( B_n(x,\psi^{\delta}))}{n}
&\ge&h_{\mu}(f,x)
\end{eqnarray}
which, by \eqref{BK_formula}, yields 
\begin{eqnarray}\label{ineq2}
\int_M  \lim_{\delta\to0} \liminf_{n\to\infty} -\frac{\log \mu( B_n(x,\psi^{\delta}))}{n}d\mu
&\ge&  \int_M h_{\mu}(f,x)d\mu
=h_{\mu}(f).
\end{eqnarray}
For another, by Proposition \ref{prop:mane}, 
\begin{eqnarray}\label{ineq3}
\int_M \lim_{\delta\to0} \limsup_{n\to\infty} -\frac{\log \mu( B_n(x,\psi^{\delta}))}{n}d\mu
= \lim_{\delta\to0}  \int_M \limsup_{n\to\infty}-\frac{\log \mu( B_n(x,\psi^{\delta}))}{n}d\mu
\le h_{\mu}(f).
\end{eqnarray}
Combining  \eqref{ineq2} and \eqref{ineq3}, we have
\begin{eqnarray*}
\int_M  \lim_{\delta\to0} \limsup_{n\to\infty} -\frac{\log \mu( B_n(x,\psi^{\delta}))}{n}d\mu&=&\int_X \lim_{\delta\to0} \liminf_{n\to\infty} -\frac{\log \mu( B_n(x,\psi^{\delta}))}{n}d\mu\\[2mm] &=&  \int_X h_{\mu}(f,x)d\mu
= h_{\mu}(f).
\end{eqnarray*}
Then together with \eqref{ineq1},  for $\mu$-a.e. $x\in M,$ it holds that $$\lim_{\delta\to0}\limsup_{n\to\infty}\ -\frac{\log \mu(B_n(x,\psi^{\delta}))}{n}=\lim_{\delta\to0}\liminf_{n\to\infty} -\frac{\log \mu( B_n(x,\psi^{\delta}))}{n}=h_{\mu}(f,x).$$

\end{proof}

\noindent {\it Proof of Theorem \ref{thm:dim_formula}:} 
Denote $\alpha=\min\{r-1, 1\}$ and  $L=\max_{ x\in I}\{|f'(x)|,1\}$.  Then $f$ is $C^{1+\alpha},$ i.e.,  for some constant $K>0,$ $|f'(x)-f'(y)|\le K|x-y|^\alpha,\ \forall x,y\in I.$ Take  small $\gamma\in (0,1),$ and for any $x\in I \backslash\Sigma_f$ let $\rho(x)=(\frac{\gamma|f'(x)|}{K})^{\frac{1}{\alpha}}.$ Then 
$$(1-\gamma) |f'(x)|\le |f'(y)|\leq(1+\gamma)|f'(x)|, \quad \forall\,y\in B(x,\rho(x)),$$
which implies that $f|_{B(x,\rho(x))}$ is a local diffeomorphism.

By the definition of Lyapunov exponent, for $\mu$-a.e. $x\in I$,
$$\int_M\log|f'(x)|d\mu(x)=\lim_{n\to \infty}\frac{1}{n}\sum_{i=0}^{n-1} \log|f'(f^ix)|=\lambda(\mu),$$
or equivalently, 
\[\lim_{n\to \infty}\frac{1}{n}\sum_{i=0}^{n-1} \log\big|(f^{-1}\mid_{f(B(f^{i}x, \rho(f^ix)))})'(f^{i+1}x)\big|=-\lambda(\mu).\]
Thus, there exists $N(x)\in\mathbb{N}$ 
such that for any $n\ge N(x),$
\begin{eqnarray}\label{Lya_ineq1}
\prod_{i=0}^{n-1} |f'(f^ix)|\in \big(e^{n(\lambda(\mu)-\gamma)}, \,e^{n(\lambda(\mu)+\gamma)}\big),
\end{eqnarray}
or equivalently
\begin{eqnarray}\label{Lya_ineq2}
\prod_{i=0}^{n-1}\big|(f^{-1}\mid_{f(B(f^{i}x, \rho(f^ix)))})'(f^{i+1}x)\big|\in\big(e^{n(-\lambda(\mu)-\gamma\big)}, \,e^{n(-\lambda(\mu)+\gamma)}).
\end{eqnarray}
Let  $ D(x)=\min_{0\le i\le N(x)}\{ \rho(f^ix),1\}$.   Note that the integrability of $\log|f'(x)|$ (and thus $\log \rho(x)$) implies $\mu(\Sigma_f)=0,$ which in turns yields $D(x)>0$ for $\mu$-a.e. $x\in I.$

Now, for any small $\delta\in(0,1)$,  denote 
$A_{\delta}=\{x\in I: \rho(x)\le \delta\},$ and
for any $n\in \mathbb{N}$, define \begin{eqnarray*}
\Phi_n(x)=\prod_{0\le i<n: f^ix\in A_{\delta}}\,\,\rho(f^ix).
\end{eqnarray*}
Then $\delta\Phi_n(x)\le\rho^\delta(f^ix), 0\le i<n.$
By the Birkhoff ergodic theorem, for $\mu$-a.e. $x\in I,$
\begin{eqnarray}\label{Birkhoff}
\frac{\log\Phi_n(x)}{n}\to\int_{A_{\delta}} \log\rho(z) d\mu(z),\quad n\to \infty.
\end{eqnarray}

Now, for any $n\in\N$ let
\begin{eqnarray*}
\tilde r_n&=&D(x)L^{-N(x)} \delta \Phi_n(x)(1+\gamma)^{-n}e^{-n(\lambda(\mu)+\gamma)},\\
\hat r_n&=&(1-\gamma)^{-n}e^{-n(\lambda(\mu)-\gamma)}.
\end{eqnarray*}
Obviously, $\tilde r_n\to0, \hat r_n\to 0$ as $n\to\infty.$
\begin{claim}\label{claim:contain}
 For $\mu$-a.e. $x\in M$ and $n\ge N(x),$ 
\begin{eqnarray}\label{inclusion}
B(x,\tilde r_n)\subset B_n(x,\rho^{\delta})
\subset B(x,\hat r_n).
\end{eqnarray}
\end{claim}
We postpone  the proof of Claim \ref{claim:contain}  for the time being  and proceed to finish the proof of Theorem \ref{thm:dim_formula}.
By \eqref{inclusion} we have
\begin{eqnarray*}
\overline{\dim}(x)
= \limsup_{n\to \infty} \frac{\log\mu(B(x,\hat r_n))}{\log\hat r_n}\le   \limsup_{n\to \infty} \frac{\log\mu(B_n(x,\rho^{\delta}))}{-n} \frac{1}{\log(1-\gamma)+\lambda(\mu)-\gamma}. \end{eqnarray*}
Then Lemma \ref{lem:int BK},  together with  the arbitrariness of  $\delta$ and $ \gamma$,   yields
\begin{eqnarray*}
\overline{\dim}(x)\le   \frac{h_{\mu}(f,x)}{\lambda(\mu)}. \end{eqnarray*}
Since $\mu$ is ergodic, it follows from Propositions \ref{prop:young} and  \ref{prop:BK} that 
\begin{eqnarray}\label{upper}
\overline{\dim}(\mu)\le\dfrac{h_\mu(f)}{\lambda(\mu)}.
\end{eqnarray}
On the other hand, \eqref{inclusion} also yields 
\begin{eqnarray*}
\underline{\dim}(x)&=&
 \liminf_{n\to \infty} \frac{\log\mu(B(x,\tilde r_n))}{\tilde r_n}\\
&\ge& \liminf_{n\to \infty} \frac{\log\mu(B_n(x,\rho^{\delta}))}{-n} \frac{1}{\frac{\log\Phi_n(x)}{n}+\log(1+\gamma)+\lambda(\mu)+\gamma}\\
&=& \liminf_{n\to \infty} \frac{\log\mu(B_n(x,\rho^{\delta}))}{-n}\cdot\frac{1}{\int_{A_{\delta}} \log\rho(z) d\mu(z)+\lambda(\mu)+\gamma},
 \end{eqnarray*}
where the last equality is from \eqref{Birkhoff}.
Note that $\mu(A_{\delta}) \to 0$ as  $\delta\to 0.$ Thus, $\int_{A_{\delta}} \log\rho(z) d\mu(z)$
can be arbitrarily small by choosing  small $\delta\in(0,1)$.  Again, applying Lemma \ref{lem:int BK}, Propositions \ref{prop:young} and \ref{prop:BK}, together with  the arbitrariness of $\delta$ and $\gamma$,  we have
\begin{eqnarray}\label{lower}
\underline{\dim}(\mu)\ge   \frac{h_{\mu}(f)}{\lambda(\mu)}. \end{eqnarray}
Combining  \eqref{lower} and \eqref{upper}, the proof of Theorem \ref{thm:dim_formula} is concluded. 

\bigskip

\noindent{\it Proof of Claim \ref{claim:contain}:}
To prove  $\tilde{B}_n(x)\subset B_n(x,\rho^{\delta}),$ we only need to show
\begin{eqnarray}\label{induction}
f^i(\tilde{B}_n(x)) \subset B(f^ix,\rho^{\delta}(f^ix)), \quad0\le i< n.
\end{eqnarray} 
First, it is not hard to see that \eqref{induction} holds for $0\le i\le N(x).$
Now, assume that for $N(x)<i\le n,$ \eqref{induction} holds for $j=0,1,\cdots,i-1.$ Then for any $y\in\tilde B_n(x),$ by \eqref{Lya_ineq1} we have
\begin{eqnarray*}
\prod_{j=0}^{i-1}\big|f'(f^jy)\big|\le\Big(\prod_{j=0}^{i-1}\big|f'(f^jx)\big|\Big)(1+\gamma)^i\le e^{i(\lambda(\mu)+\gamma)}(1+\gamma)^i\le e^{n(\lambda(\mu)+\gamma)}(1+\gamma)^n.
\end{eqnarray*}
Thus, $$f^i(\tilde B_n(x))\subset B(f^ix,D(x)L^{-N(x)}\delta\Phi_n(x))\subset B(f^ix,\delta\Phi_n(x))\subset B(f^ix,\rho^\delta(f^ix)).$$  
Then the induction establishes \eqref{induction}.

To establish $B_n(x,\rho^{\delta})
\subset\hat{B}_n(x),$ note that 
for any  $y\in f^{n}(B_n(x,\rho^{\delta})),$
\[f^iy\in B(f^ix,\rho(f^ix)),\ 0\le i\le n,\]
and hence for $0\le i\le n$, 
\[\big|f^{-1}\mid_{f(B(f^{i}(x), \rho(f^ix))})'(f^{i+1}y)\big|\le(1-\gamma)^{-1} \big|f^{-1}\mid_{f(B(f^{i}(x), \rho(f^ix))})'(f^{i+1}x)\big|.\]
Then  \eqref{Lya_ineq2} yields
\begin{eqnarray*}
\prod_{0\le i< n}|(f^{-1}\mid_{f(B(f^{i}(x), \rho(f^ix))})'(f^{i+1}y)|
\le  (1-\gamma)^{-n} e^{-n(\lambda(\mu)-\gamma)}.
\end{eqnarray*}
By noting that $f^n|_{B_n(x,\rho^\delta)}$ is a local diffeomorphism, we have  
$$B_n(x,\rho^{\delta})\subset \hat{B}_n(x).$$
This completes the proof of Claim \ref{claim:contain}.
\hfill $\Box$

\begin{proof}[Proof of Theorem \ref{thm3}]  
	
First, note that the Lyapunov exponent $\lambda(\nu)=\int \log|f'(x)|d\nu(x)$ is continuous on $\cM_{\bar\eta,erg}(f)$.  Thus, suppose that $\{\mu_i\}\subseteq\mathcal M_{\bar \eta, erg}(f) $ and  $\mu_i$ converge to some hyperbolic $\mu\in\mathcal M_{\bar \eta, erg}(f)$.  We have that $\{\mu_i\}$ are hyperbolic for all $i$ large such that $\lambda(\mu_i)>0$ ({\it resp.} $<0$) if $\lambda(\mu)>0$ ({\it resp.} $<0$).  Assume $\lambda(\mu)<0.$ In this case,  every one of  $\mu$ and $\mu_i$ ($i\gg1$)   is  supported  at  some contracting periodic orbit, and hence the corresponding $\dim(\mu)$ and $\dim(\mu_i) $  are always zero, and Theorem \ref{thm3} is proved.  Now,  assume $\lambda(\mu)>0,$ and hence $\lambda(\mu_i)>0, i\gg1.$ By Theorem \ref{thm:dim_formula}, 
the dimension  formula \eqref{dim_formula} applies for all $\mu, \mu_i (i\gg1).$ 
Since the continuity of metric entropy (by Theorem \ref{thm2}) and the Lyapunov exponent holds on $\cM_{\bar\eta,erg}(f),$
we have $\limsup_{i\to\infty}\dim(\mu_i)\le\dim(\mu).$  Thus, the upper semi-continuity of $\dim(\cdot)$ on $\cM_{\bar\eta,erg}(f)$ is obtained. 
\end{proof}

\section{An example of interval maps}\label{example}

In this section, for each $r\in(1,\infty),$ we construct a $C^r$  interval map $f$ which admits  ergodic measures with positive entropy as  non upper semi-continuity points of the metric (or folding entropy).  
Before describing the example, we recall a quantitative proposition  about the (one-sided) shift that will be used in  our discussion. 

\begin{pro}[see Theorem 4.26 and its Remark in \cite{walters82}]\label{symbolic dynamics}
Consider the $k$-full shift $(\Sigma_k, \sigma)$. For any probability vector $(p_0,\cdots,p_{k-1}),$ the Markov measure associated with the $(p_0,\cdots,p_{k-1})$-shift is ergodic with the entropy equals $\sum\limits_{i=0}^{k-1}-p_i\log p_i.$ In particular, the metric entropy of $(\frac{1}{k},\cdots,\frac{1}{k})$-shift achieves the topological entropy of the full shift $(\Sigma_k, \sigma)$ which equals to $\log k.$
\end{pro}

By Proposition \ref{symbolic dynamics},  given any full shift with topological entropy $h,$ any real number in $(0,h]$ can be achieved by an ergodic invariant measure with full support.

\setlength{\unitlength}{1.2mm}
 \begin{center}
 \begin{figure}[h]
 \begin{picture}(180,70)
\thicklines
\put(25,5){\line(1,0){63.5}}
\put(25,5){\line(0,1){63.5}}
\thinlines
\qbezier[150](25,5)(56.75,36.75)(88.5,68.5)
\qbezier(25,5)(27.2,4.8)(27.75,6)
\qbezier(27.75,6)(27.85,7)(28,8)
\qbezier(28,8)(29.25,28)(30.5,48)

\qbezier(30.5,48)(30.65,50)(30.8,52)
\qbezier(30.8,52)(34.025,83)(37.25,52)
\qbezier(37.25,52)(37.35,50)(37.45,48)

\qbezier(37.5,48)(38.75,28)(40,8)

\qbezier(40,8)(40.25,6)(41,5)
\qbezier(41,5)(41.3,5.1)(41.5,7)
\qbezier(41.5,7)(41.9,9.3)(42.5,9)

\qbezier[50](25,8)(32.5,8)(40,8)
\qbezier[110](25,9)(45,9)(66.5,9)

\qbezier(28,5)(28,6)(28,8)
\put(30.5,5){\line(0,1){43}}
\put(37.45,5){\line(0,1){43.5}}
\qbezier(40,5)(40,7.5)(40,8)

\qbezier[20](25,48)(28,48)(30.5,48)
\qbezier[15](25,26)(27,26)(29,26)

\put(24.3,47.5){\qihao{$\bullet$}}
\put(15,46){\qihao{$x_*+a\lambda$}}

\put(24.3,25.5){\qihao{$\bullet$}}
\put(15,24){\qihao{$x_*+\frac{a\lambda}{2}$}}

\qbezier[10](42.5,5)(42.5,7)(42.5,9)

\qbezier(43.5,9)(43.8,9.2)(44,10)
\qbezier(44,10)(44.2,10.5)(44.5,10)
\qbezier(44.5,10)(44.55,9.2)(45,9)

\qbezier(45,9)(45.4,9.1)(45.5,10)
\qbezier(45.5,10)(45.7,10.5)(46,10)
\qbezier(46,10)(46.05,9.2)(46.5,9)

\qbezier(47.5,9)(47.9,9.1)(48,10)
\qbezier(48,10)(48.2,10.5)(48.5,10)
\qbezier(48.5,10)(48.55,9.2)(49,9)

\qbezier[3](46.5,9.5)(46.7,9.5)(47.1,9.5)

\qbezier[2](42.7,9.3)(42.9,9.3)(43.1,9.3)

\qbezier(52.5,9)(53.1,9.4)(53.5,11)
\qbezier(53.5,11)(54.2,13)(54.9,11)
\qbezier(54.9,11)(55.3,9.2)(56,9)

\qbezier(56,9)(56.9,9.3)(57,11)
\qbezier(57,11)(57.5,13)(58.3,11)
\qbezier(58.3,11)(58.8,9.4)(59.5,9)

\qbezier(60.5,9)(61.1,9.4)(61.5,11)
\qbezier(61.5,11)(62.2,13)(62.9,11)
\qbezier(62.9,11)(63.3,9.1)(64,9)

\qbezier[5](59,10)(59.5,10)(60.5,10)

\put(28,3.5){\line(0,1){1}}
\put(30.5,3.5){\line(0,1){1}}
\put(28,4){\line(1,0){0.5}}
\put(30,4){\line(1,0){0.5}}
\put(28.3,3.2){\qihao$I_1$}

\put(37.45,4){\line(1,0){0.5}}
\put(39.5,4){\line(1,0){0.5}}
\put(37.45,3.5){\line(0,1){1}}
\put(40,3.5){\line(0,1){1}}
\put(37.8,3.2){\qihao$I_2$}

\put(43.5,5){\line(0,1){4}}
\put(49,5){\line(0,1){4}}
\put(43.5,3.5){\line(0,1){1}}
\put(49,3.5){\line(0,1){1}}
\put(43.5,4){\line(1,0){1}}
\put(48,4){\line(1,0){1}}
\put(45,3.2){\qihao$J_k$}

\put(52.5,5){\line(0,1){4}}
\put(64,5){\line(0,1){4}}
\put(52.5,3.5){\line(0,1){1}}
\put(64,3.5){\line(0,1){1}}
\put(52.5,4){\line(1,0){2}}
\put(62,4){\line(1,0){2}}
\put(57.5,3.2){\qihao$J_1$}

\qbezier(64,9)(66,9)(67,9.3)
\qbezier(67,9.3)(69,10)(70.8,13)
\qbezier(70.8,13)(85,41)(88.6,68.5)

\qbezier[120](25,68.5)(50,68.5)(88.5,68.5)
\qbezier[120](88.5,5)(88.5,35)(88.5,68.5)

\put(24.4,8.65){\qihao$\bullet$}
\put(21.5,9.75){\qihao$x_0$}

\put(41.9,4.5){\qihao$\bullet$}
\put(41,3.2){{\qihao$z_0$}}

\put(27.3,4.5){\qihao$\bullet$}
\put(27,1){\qihao$x_*$}

\put(39.3,4.5){\qihao$\bullet$}
\put(38.7,1){\qihao$x_*^\p$}

\put(23.5,9){\line(1,0){1}}
\put(23.5,5){\line(1,0){1}}
\put(24,5){\line(0,1){1}}
\put(24,8){\line(0,1){1}}
\put(22.5,7){\qihao$I_0$}

\put(90,3){$I$}
\put(22,68){$I$}
\end{picture}
\caption{Accumulation of small horseshoes}\label{fig:accumulation of small horseshoes} 
\end{figure}
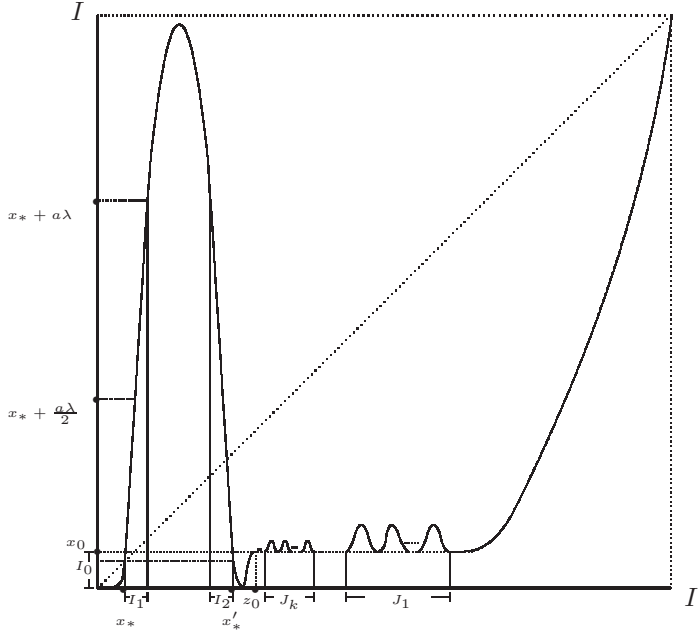
\end{center}

\subsection{Example description}
In this subsection, we describe in detail how the interval map $f: I\rightarrow I$ is constructed on  $I=[0,1]$; see Figure \ref{fig:accumulation of small horseshoes} for its qualitative depiction.
In Figure \ref{fig:accumulation of small horseshoes}, $I_1$ and $I_2$ are two subintervals of $I$ on which $f$ acts linearly with slope $\lambda$ such that
$$f(I_i)\supseteq I_1\cup I_2, \quad i=1,2.$$ Let $$\Lambda=\bigcap\limits_{i=0}^{\infty}f^{-i}(I_1\cup I_2).$$ Then $(f,\Lambda)$ is uniformly expanding  and conjugate to a (one-sided) full shift of two symbols.  Thus, the topological entropy $h_{\top}(f|_{\Lambda})=\log2$.

In the following, for any subinterval $J\subset I$, by $|J|$ we mean the length of $J$.  Any $y\in I$ stands for either a point in $I$ or a real number when  $I$ is considered as a subinterval of $\mathbb{R}$.
 Let $x_{*}$ and  $x_*^\p$ be the left and right end point of $I_1$ and $I_2$,  respectively. We further require that $x_*$  is a fixed  point of $f$ such that  $f(x_*^\p)=f(x_*)=x_*$.
For simplicity, set $a=x_*=|I_1|=|I_2|$ and $d_{H}(I_1, I_2)= 3a$.  Also, let  $\lambda$ be large enough so that  \[I_1\cup I_2 \subset  [x_*, x_*+\dfrac{a\lambda}{2}].\]
Since $(f,\Lambda)$ conjugates to a one-sided 2-full shift, by Proposition \ref{symbolic dynamics}, for any $c\in (0, c_0)$ where $c_0=\min\{\log 2, \frac{1}{r}\log \lambda\},$ we can find  $\mu\in\cM_{erg}(f)$ supported on $\Lambda$ such that  
\begin{eqnarray}\label{h_mu}
h_\mu(f)=c<\frac{1}{r}\log\lambda.
\end{eqnarray}
 Take $\delta_0=a/(2\lambda).$  Observing that $x_*\in \Lambda$, we can choose  a generic point $x_0\in \supp\mu\cap [x_*, \delta_0/2]$ of $\mu$ in the sense that  $$\dfrac{1}{n}\sum\limits_{i=0}^{n-1}\delta_{f^i(x_0)}\rightarrow\mu, \quad  \text{as}\,\,n\rightarrow\infty,$$
in the weak*-topology. 
As denoted in Figure \ref{fig:accumulation of small horseshoes}, arrange $z_0$ to be a preimage of $x_0$ lying on the right side of $x_*'$ such that 
$|x_*'-z_0|=a.$
Also, we can require that $f$ remains linear with slope $\lambda$ in the $\delta_0$-neighborhood of $I_1\cup I_2$
since $d_{H}(I_1, I_2)= 3a$.
	
\begin{lemma}\label{small cover} Let $\delta_1=x_0-x_*$. Then there exist $N_1, N_2\in \mathbb{N}$ such that
	$$f^{N_1}([x_0-\delta_1, x_0]) \supset [z_0-a, z_0+3a],\quad  f^{N_2}([x_0, x_0+\delta_1]) \supset [z_0-a, z_0+3a].	$$
\end{lemma}

\begin{proof}
Since $x_*$ is a fixed point, the expanding property of $f$ on $I_1$ gives certain $N_1\in \mathbb{N}$ such that  $$f^{N_1}([x_0-\delta_1, x_0])\supset  f(I_1)\supset  [z_0-a, z_0+3a].$$
Then by the continuity of $f,$ there exists  $\delta\in (0,\delta_1)$ such that $$f^{N_1}([x_0-\delta_1+\delta, x_0])\supset  [z_0-a, z_0+3a].$$
Recall that $x_0$ is a generic point of $\mu$.  Since  $\mu([x_0-\delta_1, x_0-\delta_1+\delta])>0$, there exists $t\in \mathbb{N}$ such that $$f^{t}(x_0)\in [x_0-\delta_1, x_0-\delta_1+\delta] .$$

Now, consider the $f$-iterates of $[x_0, x_0+\delta_1]$ successively cut by the interval  $[x_*-\delta_0, x_*'+\delta_0]$. Then there exists $t_0>0$ such that for each $i>t_0,$ $f^i([x_0,x_0+\delta_1])$ contains an interval $R_i$ of length $\delta_0$ with $f^i(x_0)$ being one of its end points. In particular, let $t>t_0.$  
Now, we proceed  according to the cases where $f^t(x_0)$ is the left  and right end point of
$R_t,$ respectively.
If $f^t(x_0)$ is the left end point, we have $R_t\supset [x_0-\delta_1+\delta, x_0],$ and hence $$f^{t+N_1}([x_0, x_0+\delta_1])\supset f^{N_1}(R_t)\supset [z_0-a, z_0+3a].$$
Then  $N_2=t+N_1$ is as desired;
If  $f^t(x_0)$	is the right end point of $R_t$, then $$R_t\supset [x_0-\delta_1, f^tx_0].$$ 	
Once more, using the expanding property of $f$ on $I_1$, there exists $t'\in \mathbb{N}$, such that
\[f^{t'}([x_0-\delta_1, f^tx_0])\supset  [z_0-a, z_0+3a].\]
In this case, let $N_2=t+t'$. So the proof is finished.
\end{proof}
By Lemma \ref{small cover} and the continuity of $f$, there exists $\eta>0$ such that for any $x\in [x_0-\eta, x_0+\eta]$,
\begin{eqnarray}\label{interval_contain}
f^{N_1}([x-\delta_1, x]) \supset [z_0, z_0+2a],\quad  f^{N_2}([x, x+\delta_1]) \supset [z_0, z_0+2a].	
\end{eqnarray}
Since $x_0$ is a generic point  and $\mu([x_0-\eta, x_0+\eta])>0$, there exist $n_1<n_2<\cdots<n_k<\cdots$ such that $f^{n_k}(x_0)\in [x_0-\eta, x_0+\eta] $. 	
Now, let $\{J_n\}_{n\geq1}$ be a sequence of disjoint subintervals of $I$  accumulating to $z_0$ from the right with the following:
\begin{itemize}
\item[(i)] On each interval $J_k,$ $f|_{J_k}=A_k^r\cos\omega_k(x-c_k)+(x_0+A_k^r),$ where $\{c_k\}_{k\geq1}$ is a sequence of real numbers decreasing to  $z_0,$ such that
\begin{eqnarray*}
A_k=\big{(}\dfrac{\delta_0}{2}\lambda^{-n_k}\big{)}^{\frac{1}{r}},\quad  \omega_k=\dfrac{L}{A_k},
\end{eqnarray*}
where $L\ge \lambda$ is chosen to satisfy  $\max_{ 1\le i\le r}\sup\nolimits_{x\in I}|f^{(i)}(x)|\le L^r$;

\item[(ii)] On each interval $J_k,$ $f$ oscillates $M_k$ times, 
where $$M_k=\frac{ L\gamma_0}{2\pi  k^2}\big{(}\frac{2}{\delta_0}\lambda^{n_k}\big{)}^{\frac{1}{r}}$$
and $\gamma_0$ is a small real number.
\end{itemize}
Obviously, $x_0\in f(J_{k})$.
Also, by (i) and (ii), we see that
$|J_k|=\dfrac{2\pi M_k}{\omega_k}=\dfrac{\gamma_0}{k^2}.$ Thus, by choosing a small $\gamma_0>0$, we can have
\[\sum\limits_{k=1}^{\infty}|J_k|=\gamma_0\sum\limits_{k=1}^{\infty}\frac{1}{k^2}<a.\]

Make \[\bigcup_{k\ge 1 } J_k \subset [z_0, z_0+2a].\]
Outside the intervals $\{I_i\}_{i=1,2}$ and $\{J_k\}_{k\geq1},$ we extend $f$ in a smooth way as depicted in Figure \ref{fig:accumulation of small horseshoes}.

\subsection{Analysis of Example}\label{property of f}
In this subsection, we show that $\mu$ is approximated by the ergodic measures supported on a sequence of horseshoes with topological entropy having uniform gap  from the metric entropy of $\mu$. Hence,  the  metric entropy is not upper semi-continuous at $\mu$.

For an interval map $g$ and  integer $\ell\geq2,$  by an $\ell$-horseshoe of $g$ we mean a family of  disjoint closed intervals $(K_1,\cdots,K_\ell)$ such that \[g(K_i)\supseteq K_j, \quad \forall\, i, \,j\in\{1, \cdots,\ell\}.\] 
For an interval $K,$ we say $g$ admits an $\ell$-horseshoe on $K$ if $\cup_{1\le i\le \ell} K_i\subset K.$ 
Recall that an $\ell$-horseshoe is conjugated to the one-sided shift of $\ell$ symbols.

\begin{lemma}\label{creation of horsehsoes}
For each $k\geq1,$ either $f^{n_k+N_1+1}(J_k)\supseteq J_k$, or $f^{n_k+N_2+1}(J_k)\supseteq J_k$.
\end{lemma}

\begin{proof}For each $k\geq1$,  since $x_0\in f(J_k),$ we have $f^i(x_0)\in f^{i+1}(J_k)$ for any $i\geq1$. Note that for $0\le i\le n_k$,  $$|f^{i+1}(J_k)|=\lambda^i |f(J_k)|=\lambda^i \cdot 2A_k^r\le \delta_0$$  and $|f^{n_k+1}(J_k)|= \delta_0$. Moreover,  $f^{n_k}(x_0)\in [x_0-\eta, x_0+\eta] $, $\delta_1<\delta_0$. By \eqref{interval_contain} we have
\[f^{N_1}(f^{n_k+1}(J_k))\supset J_k\quad \text{or}\quad f^{N_2}(f^{n_k+1}(J_k))\supset J_k.\]
\end{proof}

In this following, we only consider the case $f^{n_k+N_1+1}(J_k)\supseteq J_k$ since the other one is similar. Then  $f^{n_k+N_1+1}|_{J_k}$ admits  a $2M_k$-horseshoe. Let $$\Lambda_k=\bigcup\limits_{j=0}^{n_k+N_1}f^{j}\Big{(}\bigcap\limits_{i\geq0}f^{-(n_k+N_1+1)i}(J_k)\Big{)}.$$ By Proposition \ref{symbolic dynamics}, for each $k,$ there exists an ergodic measures $\nu_k$ supported on $\Lambda_k$ such that
\begin{eqnarray*}
h_{\nu_k}(f)=h_{\top}(f|_{\Lambda_k})=\dfrac{\log (2M_k)}{n_k+N_1+1}.
\end{eqnarray*}
Thus,
\begin{eqnarray}\label{entropy sequence}
h_{\nu_k}(f)\to\frac{1}{r}{\log\lambda},\quad \quad k\rightarrow\infty.
\end{eqnarray}

Now, we show that $\nu_k\to \mu$ as $k\to \infty$.  Given any continuous functions $\varphi_1, \cdots \varphi_s$ on $X$, for any $\vep>0$, there exists $\gamma>0$ such that for any  $x,y\in I$ satisfying $|x-y|<\gamma$,
\[|\varphi_i(x)-\varphi_i(y)|<\frac{\vep}{2},\quad i=1,\cdots,s.\]  Let 
\[t_k=\ln({\gamma}/{2A_k^r})/\ln \lambda=\frac{\ln\gamma-\ln \delta_0}{\ln \lambda}+n_k.\] Then $$\lim_{k\to \infty}\frac{t_k}{n_k}=1\quad \text{and}\quad |f^{j+1}(J_k)|\le \gamma,\quad \forall\,0\le j\le t_k,$$
which implies by letting $k$ large that 
\begin{eqnarray*}
\lim_{k\to \infty}\Big{|}\int \varphi_i(x)d\nu_k(x)-\frac{1}{n_k}{\sum_{0\le j\le n_k-1} \varphi_i(f^j(x_0))}\Big{|}\le \vep,\quad i=1,\cdots,s.
\end{eqnarray*}
 Since $(1/n_k)\sum_{0\le j\le n_k-1}\delta_{f^jx_0} \to \mu$ as $k\to \infty$, together with the arbitrariness of $\varphi_1,\cdots, \varphi_s$ and $\vep,$ we obtain the convergence of $\{\nu_k\}$ to $\mu$.

By \eqref{h_mu} and \eqref{entropy sequence},
\begin{eqnarray*}
\lim_{k\to\infty}h_{\nu_k}(f)=\frac{1}{r}\log\lambda>h_\mu(f),
\end{eqnarray*}
Therefore,  the metric entropy and hence the folding entropy, is not upper semi-continuous at $\mu.$  
In the following, we  show that  
$\{\nu_n\}$ does not admit  uniform degenerate  rate. This demonstrate that the condition of  uniform degenerate  rate in Theorem \ref{thm1} is sharp.

Recall that $\Sigma_f=\{x\in I: f'(x)=0\}.$ Note that in our example, $z_0\in\Sigma_f.$
Given a decreasing sequence of neighborhoods $\mathcal V=\{V_m\}_{m\ge1}$ of $\Sigma_f.$
For any fixed $m\ge1,$
we have $J_k\subseteq V_m$ for all $k$ sufficiently large.   Let $y_k\in J_k$ be a generic point of $\nu_k,$ i.e.,  $\dfrac{1}{t}\sum\limits_{i=0}^{t-1}\delta_{f^ty_k}\rightarrow\nu_k$, $t\rightarrow\infty.$
Note that   $$|f'(x)|\le A_k^{r}\omega_k,\  \forall x\in J_k.$$
Then for any fixed $V_m$  and all large $k,$  we have
\begin{eqnarray*}
\displaystyle\int_{V_m}\log|f'(x)|d\nu_k&\le&\dfrac{1}{n_k+N_1+1}\log(A_k^{r}\omega_k)\\
&=&\dfrac{1}{n_k+N_1+1}\log(\frac{\delta_0}{2}\lambda^{-n_k})^{\frac{r-1}{r}}\\
&=&\dfrac{\frac{r-1}{r}\log(\frac{\delta_0}{2})}{n_k+N_1+1}+\dfrac{\frac{r-1}{r}n_k\log\lambda^{-1}}{n_k+N_1+1},
\end{eqnarray*}
which, as $ k\rightarrow\infty,$ goes to $\frac{r-1}{r}\log\lambda^{-1}.$ Thus, for any $m\ge1,$
\[\Big|\displaystyle\int_{V_m}\log|f'(x)|d\nu_k\Big|>
\dfrac{r-1}{2r}\log\lambda,\quad {\text{for all}}\ k\ {\text{sufficiently large}}.\]
That is,  the sequence $\{\nu_k\}_{k\ge1}$ does not have uniform degenerate rate. This implies that the  uniform degenerate rate condition  in Theorem \ref{thm1} and Theorem \ref{thm2} is sharp.

\section*{Acknowledgement}
G. Liao was partially supported by NSFC (11701402,  11790274),  BK 20170327 and IEPJ. S. Wang was partially supported by NSFC (11771026, 11471344) and acknowledges the PIMS-CANSSI postdoctoral  fellowship.

 \bibliography{myref}
 \bibliographystyle{amsplain}

\end{document}